\documentclass{amsart}

\usepackage{mathrsfs}
\usepackage[all,hyperref]{paper_diening}

\usepackage{tikz}
\usetikzlibrary{calc}

\usepackage{mathtools}

\usepackage{algorithm}
\usepackage{algpseudocode}

\usepackage{graphicx}
\usetikzlibrary{arrows.meta,positioning,shapes.geometric}
\usepackage{adjustbox}

\providecommand{\tria}{\mathcal{T}}

\providecommand{\dx}{\,\mathrm{d}x}
\providecommand{\dr}{\,\mathrm{d}r}
\providecommand{\dt}{\,\mathrm{d}t}

\providecommand{\ds}{\,\mathrm{d}s}
\providecommand{\dz}{\,\mathrm{d}z}
\providecommand{\dy}{\,\mathrm{d}y}

\usepackage[normalem]{ulem}

\usepackage{tikz}
\usepackage{pgfplots}
\usepackage{pgfplotstable} 
\pgfplotsset{
  width=.65\linewidth,
  axis background/.style={fill=black!5!white},
  grid style={densely dotted,semithick},
  legend style={
    legend columns=1,
    legend pos=outer north east
  },
  compat=newest 
}
\pgfplotscreateplotcyclelist{MyColors2}{%
    {red,mark = *,every mark/.append style={solid,scale=0.6,fill=red}},
	{teal,mark = x,every mark/.append style={solid,scale=0.6,fill=teal}},
    {green,mark = diamond*,every mark/.append style={solid,scale=0.6,fill=green}},    
    {blue,mark = square*,every mark/.append style={solid,scale=0.6,fill=blue}},    
    {blue,mark = square*,every mark/.append style={solid,scale=0.6,fill=blue}},    
    {blue,mark = square*,every mark/.append style={solid,scale=0.6,fill=blue}},    
    }
    
\pgfplotscreateplotcyclelist{MyColorsExp2}{%
    {red,solid,mark=*,every mark/.append style={solid,scale=0.6,fill=red}},
    {teal,dashed,mark=square*,every mark/.append style={solid,scale=0.6,fill=teal}},
    {green!60!black,dashdotted,mark=triangle*,every mark/.append style={solid,scale=0.6,fill=green!60!black}},
    {blue,dotted,mark=diamond*,every mark/.append style={solid,scale=0.6,fill=blue}},
    {orange,solid,mark=x,every mark/.append style={solid,scale=0.7}},
    {violet,dashed,mark=+,every mark/.append style={solid,scale=0.7}},
    {brown,dashdotted,mark=pentagon*,every mark/.append style={solid,scale=0.6,fill=brown}},
    {cyan!60!black,dotted,mark=o,every mark/.append style={solid,scale=0.6}},
    {magenta,solid,mark=star,every mark/.append style={solid,scale=0.7}},
    {black,dashed,mark=triangle,every mark/.append style={solid,scale=0.7}}
}

\pgfplotscreateplotcyclelist{MyColors}{%
    {red,mark = *,every mark/.append style={solid,scale=0.6,fill=red}},
	{teal,mark = x,every mark/.append style={solid,scale=0.6,fill=teal}},
    {blue,mark = square*,every mark/.append style={solid,scale=0.6,fill=blue}},    
    {dotted,red,mark = *,every mark/.append style={solid,scale=0.6,fill=red}},
	{dotted,teal,mark = x,every mark/.append style={solid,scale=0.6,fill=teal}},
    {dotted,blue,mark = square*,every mark/.append style={solid,scale=0.6,fill=blue}},
    }

\makeatletter
\@namedef{subjclassname@2020}{%
  \textup{2020} Mathematics Subject Classification}
\makeatother

\begin{document}

\author[J.~Storn]{Johannes Storn}
\address[J.~Storn]{Faculty of Mathematics \& Computer Science, Institute of Mathematics, Leipzig University, Augustusplatz 10, 04109 Leipzig, Germany}
\email{johannes.storn@uni-leipzig.de}

\keywords{randomized projection operator, Monte Carlo quadrature, load vector assembly, rough right-hand side, higher order quadrature, least-squares polynomial approximation}
\subjclass[2020]{
	65D32,  
	65D30,
	65N30, 
}

\thanks{The work of the author was supported by the Deutsche Forschungsgemeinschaft (DFG, German Research Foundation) -- SFB 1283/2 2021 -- 317210226.}

\title[Singular Zienkiewicz Tetrahedron]{The singular Zienkiewicz tetrahedron: Definition and Integration}

\begin{abstract}
We extend the two-dimensional singular Zienkiewicz element to three dimensions, leading to a novel $H^2$-conforming finite element on tetrahedral meshes based on rational shape functions. Besides the finite element construction and its conformity analysis, we develop an exact iterative integration procedure for the associated class of rational functions. The resulting formulae allow for the exact integration of the basis functions, their derivatives, and the products occurring in finite element assembly.
\end{abstract}

\maketitle

\section{Introduction}\label{sec:Intro}

Conforming finite element methods for fourth-order problems require globally $H^2$-conforming approximation spaces. On simplicial meshes, this entails continuity of the function and of its first derivatives across element interfaces and therefore leads to considerably stronger constraints than standard $H^1$-conformity. Nevertheless, several $C^1$ finite elements are available in two dimensions. Their construction requires either relatively high polynomial degrees, macroelement refinements, or non-(piecewise-)polynomial shape functions, cf.~\cite[Chap.~6]{Ciarlet2002}. In three dimensions, the situation is similar: The \v{Z}en\'{\i}\v{s}ek element \cite{Zenisek73} is a $H^2$-conforming finite element that requires the polynomial space of maximal degree nine.
Lower-degree alternatives are based on macroelement refinements, including the Alfeld split \cite{Alfeld84,FuGuzmanNeilan20} and its refinement by the Worsey--Farin split \cite{WorseyFarin87,GuzmanLischkeNeilan22}, as well as the finer Worsey--Piper split \cite{WorseyPiper88,SchumakerSorokinaWorsey09}. 
However, to the author's knowledge, extensions of the non-polynomial approach to 3D are open -- a gap that is closed by this paper's extension of the singular Zienkiewicz triangle \cite[Thm.~6.1.4]{Ciarlet2002} to three dimensions. 
In particular, we design an $H^2$-conforming finite element that uses rational polynomials as shape functions and provide an exact iterative integration formula for the resulting shape functions and their derivatives.
Our study is motivated from several perspectives.

\begin{enumerate}
\item $H^2$-conforming finite element spaces arise naturally in the conforming discretization of fourth-order problems, cf.~\cite{AinsworthParker24}. A prototypical example is the biharmonic equation. Related applications include higher-order continuum models such as strain-gradient elasticity \cite{PapanicolopulosZervosVardoulakis09} and primal formulations of phase-field problems \cite{ZhaoSchillingerXu17}. Moreover, in three dimensions the Sobolev embedding $H^2(\Omega)\hookrightarrow C^0(\overline\Omega)$ makes point evaluations continuous. This additional regularity is useful for nonlinear Hessian-dependent problems and equations involving pointwise constraints, such as discretizations of the Monge--Amp\`ere equation \cite{FengNeilan14,GallistlTran24,GallistlTran25}.\label{itm:one}
\item Smooth finite element spaces are closely related to exactly divergence-free discretizations of incompressible flow through the Stokes complex
\begin{align*}
\mathbb R\longrightarrow H^2(\Omega)\xrightarrow{\nabla}H^1(\operatorname{curl};\Omega)\xrightarrow{\operatorname{curl}}H^1(\Omega)^3\xrightarrow{\operatorname{div}}L^2(\Omega)\longrightarrow 0.
\end{align*}
In particular, the design and analysis of suitable exactly divergence-free finite elements often uses discrete analogues of this sequence and commuting diagram properties, cf.~\cite{ChenHuang24,FuGuzmanNeilan20,FalkNeilan13,JohnLinkeMerdonNeilanRebholz17,Neilan20}. 
However, unlike in the 2D case, discrete sequences with rational functions that include for example the 3D Guzm\'an--Neilan element \cite{GuzmanNeilan14} are, to the author's knowledge, yet unknown. Our studies provide a building block toward such a construction.
\item There is continued research interest in $C^r$-conforming finite elements, see for example \cite{HuLinWu24}. For polynomial constructions, enforcing higher-order continuity across interfaces typically increases the degree of the trace spaces and thereby requires additional degrees of freedom and further enrichments. This mechanism can lead to high-dimensional local finite element spaces. The idea underlying the singular Zienkiewicz triangle -- enriching a low-order polynomial space by rational functions while keeping the relevant traces in suitably low-dimensional spaces -- provides a way to avoid this cascading increase of polynomial degree and local dimension.\label{itm:three}
\end{enumerate}
Our design of an $H^2$-conforming element can be seen as an extension of the singular Zienkiewicz triangle \cite[Thm.~6.1.4]{Ciarlet2002} to three dimensions. As illustrated in \cite[Sec.~2.1]{DieningStornTscherpel24}, the singular Zienkiewicz triangle adds non-polynomial basis functions with polynomial traces to a polynomial space. This allows for the inclusion of additional degrees of freedom, ensuring global $C^1$-conformity. Inspired by this idea, we introduce in Section~\ref{sec:RatBubbles} a class of rational polynomials which are used in Section~\ref{subbsec:SingZienk}--\ref{subsec:reduced-singular-zienkiewicz-tetrahedron} to introduce a suitable finite element space, including a reduced version with fewer degrees of freedom and an extended version with improved approximation properties.
While this approach yields low-dimensional $H^2$-conforming finite elements, their practical implementation poses an additional challenge: The second derivatives of the rational shape functions are in general only in $L^\infty$, and standard polynomial quadrature rules therefore do not provide accurate results causing additional errors in the resulting finite element method. To address this difficulty, we develop in Section~\ref{sec:3d-rational-moments} an exact integration routine for the rational functions occurring in our construction. Similar to the two-dimensional routine in \cite{DieningStornTscherpel24}, the algorithm is based on a sequence of algebraic reduction identities. In three dimensions, however, the additional denominator configurations require new reduction steps. We first reduce arbitrary denominator graphs to triangular and vertex-star configurations. The latter are then treated by admissible star reductions, a descent formula based on homogeneity and the divergence theorem, and finally explicit lower-dimensional integral identities. Exact integration alone however does not yet guarantee an efficient implementation, since a direct three-dimensional extension of the two-dimensional assembly routine leads to a comparatively large geometry-dependent tensor contraction. In Section~\ref{subsec:efficient-stiffness-assembly}, we therefore exploit the affine transformation of the Hessian to separate geometry and reference-element information more efficiently. Since the associated symmetric geometry matrix has only six independent entries, its quadratic contribution can be represented by $21$ geometry coefficients. This reduces the online assembly to geometry-dependent linear combinations of precomputed reference matrices.
We conclude in Section~\ref{subsec:NumExp} with numerical experiments. In particular, we compare the exact integration procedure with numerical quadrature and illustrate the influence of inexact integration on the resulting finite element approximations.
\section{The polynomial Zienkiewicz tetrahedron}\label{sec:poly-zienkiewicz-tetrahedron}
In this section we discuss the (polynomial) Zienkiewicz finite element in $\mathbb{R}^3$, see also \cite{MingXu07}. 
Let $T=[v_0,v_1,v_2,v_3]\subset \mathbb R^3$ be a tetrahedron with vertices $v_0,\dots,v_3$ and associated  barycentric coordinates $\lambda_0,\dots,\lambda_3$. For $i=0,\dots,3$, we denote the face opposite to the vertex $v_i$ by $f_i \coloneqq \lbrace \lambda_i=0\rbrace$.
Let $\nu_i$ denote the outer unit normal on $f_i$ and set
\begin{align}\label{eq:defDii}
  d_{ii}\coloneqq \nabla\lambda_i\cdot\nu_i=-|\nabla\lambda_i|<0.
\end{align}
The polynomial Zienkiewicz tetrahedron is obtained by enriching quadratic polynomials by cubic edge functions. More precisely, let $\mathcal{P}_k(T)$ denote the space of polynomials on $T$ with maximal degree $k \in \mathbb{N}_0$. Then we define the polynomial space
\begin{align*}
Z_3^0(T) \coloneqq \mathcal P_2(T)  \oplus \operatorname{span} \lbrace \lambda_i^2\lambda_j-\lambda_j^2\lambda_i \colon 0\leq i<j\leq 3 \rbrace.
\end{align*}
\begin{remark}[Equivalent characterization]\label{rem:EquiChar}
The definition of $Z_3^0(T)$ is inspired by its 2D analog in \cite{DieningStornTscherpel24}. It can be equivalently characterized as in \cite{Ciarlet2002,MingXu07} as the space of all cubic functions $p \in \mathcal{P}_3(T)$ that satisfy for all $i=0,\dots,3$ with $J_i\coloneqq \lbrace 0,1,2,3\rbrace\setminus\lbrace i\rbrace$ and barycenters $\operatorname{mid}(f_i)$ of the faces $f_i$
\begin{align*}
0 = 6 p(\operatorname{mid}(f_i)) - 2 \sum_{j\in J_i} p(v_j) + \sum_{j\in  J_i} \nabla p(v_j) \cdot \big(v_j - \operatorname{mid}(f_i)\big).
\end{align*} 
Indeed, a direct calculation shows that any basis function in the space $Z_3^0(T)$ defined above satisfies this constraint. Equality then follows by comparing the dimension
\begin{align*}
\dim Z_3^0(T) =\dim \mathcal{P}_2(T) + 6 = 16 = \dim \mathcal{P}_3(T) - 4.
\end{align*} 
\end{remark}
The degrees of freedom in the Zienkiewicz space are the vertex Hermite data
\begin{align}\label{eq:DOFShermite}
z(v_i)\qquad\text{and}\qquad \nabla z(v_i)\qquad \text{for } i=0,\dots,3.
\end{align}
These are $4(1+3)=16$ degrees of freedom, agreeing with the dimension of $Z_3^0(T)$. The following lemma shows that $Z_3^0(T)$ together with these degrees of freedom defines a finite element in the sense of Ciarlet, cf.~\cite[Chap.~2.2.3]{Ciarlet2002}.
\begin{lemma}[Unisolvence of the polynomial Zienkiewicz tetrahedron]\label{lem:unisolvence}
The degrees of freedom in \eqref{eq:DOFShermite} are unisolvent for $Z_3^0(T)$.
\end{lemma}
\begin{proof}
Suppose $z\in Z_3^0(T)$ vanishes in all degrees of freedom; that is,
\begin{align*}
  z(v_i)=0\qquad\text{and}\qquad \nabla z(v_i)=0\qquad  \text{for all }i=0,\dots,3.
\end{align*}
By definition, there exists a $p \in \mathcal P_2(T)$ and coefficients $(c_{ij})_{0\leq i < j \leq 3} \subset \mathbb{R}$ such that 
\begin{align*}
z = p + \sum_{0\leq i<j\leq 3} c_{ij} (\lambda_i^2\lambda_j-\lambda_j^2\lambda_i).
\end{align*}
We fix an edge $e_{ij}=[v_i,v_j]$ with $0\leq i<j\leq 3$. On this edge all barycentric coordinates except $\lambda_i$ and $\lambda_j$ vanish. Hence
\begin{align*}
z|_{e_{ij}} = p|_{e_{ij}} + c_{ij} (\lambda_i^2\lambda_j-\lambda_j^2\lambda_i).
\end{align*}
The restriction $z|_{e_{ij}}$ is a cubic Hermite polynomial on the interval $e_{ij}$ with point values and tangential derivative equal to zero at the endpoints of $e_{ij}$. Consequently, we obtain $z|_{e_{ij}}=0$. Since the polynomial $p|_{e_{ij}}$ has degree at most two, whereas $\lambda_i^2\lambda_j-\lambda_j^2\lambda_i$ has a nonzero cubic part on $e_{ij}$, also the components of $z|_{e_{ij}}$ must vanish; that is, $c_{ij}=0$ and $p|_{e_{ij}}=0$. Since this holds for all six edges, all coefficients $c_{ij}$ vanish. Moreover, the quadratic polynomial $p$ vanishes on the boundary of $T$ and therefore $p=0$. This proves $z=0$.
\end{proof}

In order to investigate global continuity properties, we focus on the traces of $Z_3^0(T)$ on a fixed facet $f_i$ with $i\in \lbrace0,\dots,3\rbrace$. Define the index set $J_i\coloneqq \lbrace 0,1,2,3\rbrace\setminus\lbrace i\rbrace$. We use the restrictions of the barycentric coordinates $\mu_j \coloneqq \lambda_j|_{f_i}$ with $j\in J_i$ as barycentric coordinates on $f_i$.
\begin{lemma}[Facet trace and normal trace]\label{lem:FacetTracePoly}
For every $i=0,\dots,3$, the scalar trace of $Z_3^0(T)$ on $f_i$ is the two-dimensional polynomial Zienkiewicz space
\begin{align*}
  Z_3^0(T)|_{f_i} = Z_2^0(f_i) \coloneqq \mathcal P_2(f_i) + \operatorname{span} \lbrace  \mu_j^2\mu_k-\mu_k^2\mu_j \colon j,k\in J_i,\ j<k \rbrace.
\end{align*}
Moreover, the normal trace space consists of quadratic polynomials; that is,
\begin{align*}
  N_i^0(T) \coloneqq \lbrace (\partial_{\nu_i} z)|_{f_i} \colon z\in Z_3^0(T) \rbrace = \mathcal P_2(f_i).
\end{align*}
\end{lemma}
\begin{proof}
The statement for the scalar trace follows directly from the definition. 
It remains to characterize the normal trace. Let $q\in\mathcal P_1(f_i)$ and choose any affine extension $\widetilde q\in\mathcal P_1(T)$ with $\widetilde q|_{f_i}=q$. Define
\begin{align*}
  z_q \coloneqq \frac{\lambda_i}{d_{ii}}\widetilde q \in\mathcal P_2(T)\subset Z_3^0(T)\qquad\text{with }d_{ii}\text{ defined in \eqref{eq:defDii}}.
\end{align*}
Since $\lambda_i=0$ on $f_i$ and $d_{ii}\coloneqq \nabla\lambda_i\cdot\nu_i$, the product rule reveals
\begin{align*}
  \partial_{\nu_i} z_q|_{f_i} = \frac{1}{d_{ii}}\big((\nabla\lambda_i\cdot\nu_i)\widetilde q +\lambda_i\partial_{\nu_i}\widetilde q\big)|_{f_i} = \widetilde q|_{f_i} = q.
\end{align*}
Hence every affine polynomial on $f_i$ occurs as the normal derivative of a function in $\mathcal P_2(T)\subset Z_3^0(T)$, and therefore $\mathcal P_1(f_i)\subset N_i^0(T)$.
Let now $j\in J_i$ and set
\begin{align*}
  \psi_{ij}\coloneqq \lambda_i^2\lambda_j-\lambda_j^2\lambda_i.
\end{align*}
Since $\lambda_i=0$ on $f_i$, we obtain
\begin{align*}
  \partial_{\nu_i}\psi_{ij}|_{f_i} = -(\nabla\lambda_i\cdot\nu_i)\mu_j^2 = -d_{ii}\mu_j^2.
\end{align*}
Since $d_{ii}\neq0$, this shows $\mu_j^2\in N_i^0(T)$ for all $j\in J_i$. Thus
\begin{align*}
 \mathcal{P}_2(f_i) = \mathcal P_1(f_i) \oplus \operatorname{span}\lbrace \mu_j^2\colon j\in J_i\rbrace \subset N_i^0(T).
\end{align*}
On the other hand, every $z\in Z_3^0(T)$ is a cubic polynomial, and therefore $\partial_{\nu_i}z|_{f_i}$ is a quadratic polynomial on $f_i$. This proves $N_i^0(T)=\mathcal P_2(f_i)$.
\end{proof}
As shown in the following lemma, the global Zienkiewicz element is continuous but not continuously differentiable. In other words, it is $H^1(\Omega)$ but not $H^2(\Omega)$ conforming. 
\begin{lemma}[Continuity]\label{lem:ContPolyZienk}
Let $T^+$ and $T^-$ be two tetrahedra sharing a facet $f_0$ and let $q_+ \in Z_3^0(T_+)$, $q_- \in Z_3^0(T_-)$. If $q_-$ and $q_+$ have the same vertex Hermite data on shared vertices, we have the continuity property $q_-|_{f_0} = q_+|_{f_0}$.  
However, in general the normal derivative with respect to a normal $\nu_0 \in \mathbb{R}^3$ of $f_0$ does not coincide; that is,
\begin{align*}
(\nabla q_- \cdot \nu_0)|_{f_0} \neq (\nabla q_+ \cdot \nu_0)|_{f_0}.
\end{align*}
\end{lemma}
\begin{proof}
Let $f_0=[v_1,v_2,v_3]$ be the common facet. By the trace characterization of the polynomial Zienkiewicz tetrahedron, the restriction of $Z_3^0(T^\pm)$ to $f_0$ is the two-dimensional polynomial $Z^0_2(f_0)$, which is uniquely characterized by the restrictions of the degrees of freedom onto $f_0$, cf.~the proof of Lemma~\ref{lem:unisolvence}.
It remains to show that the normal derivatives might differ. We give an explicit example. Let $\lambda_0,\lambda_1,\lambda_2,\lambda_3$ denote the barycentric coordinates on $T^+$ with face $f_0=\lbrace \lambda_0=0\rbrace$. Set $d_{00}\coloneqq \nabla\lambda_0\cdot\nu_0 \neq 0$ and define
\begin{align*}
  z^+ \coloneqq \lambda_0(1-2\lambda_3) 
  + (\lambda_0^2\lambda_1-\lambda_1^2\lambda_0)
  + (\lambda_0^2\lambda_2-\lambda_2^2\lambda_0)
  - (\lambda_0^2\lambda_3-\lambda_3^2\lambda_0)\in Z_3^0(T^+).
\end{align*}
Since every term contains the factor $\lambda_0$, we have $z^+|_{f_0}=0$. In particular, all tangential derivatives of $z^+$ vanish on $f_0$. Moreover, writing $z^+=\lambda_0 g$, we have
\begin{align*}
  g = 1-2\lambda_3-\lambda_1^2-\lambda_2^2+\lambda_3^2\qquad\text{on }f_0.
\end{align*}
At the vertices $v_1,v_2,v_3$ this function vanishes. Hence, $\nabla z^+=g\nabla\lambda_0+\lambda_0\nabla g$ also vanishes at the vertices of $f_0$. Thus $z^+$ has zero vertex Hermite data on the shared facet $f_0$.
If we choose $q_-\coloneqq 0$ on $T^-$ and $q_+\coloneqq z^+$ on $T^+$, then $q_-$ and $q_+$ have identical vertex Hermite data on the shared vertices, but $\nabla q_-\cdot\nu_0|_{f_0}=0$, whereas
\begin{align*}
  (\nabla q_+\cdot\nu_0)|_{f_0} = (g \nabla \lambda_0 \cdot \nu_0)|_{f_0} = d_{00}\big(1-2\lambda_3-\lambda_1^2-\lambda_2^2+\lambda_3^2\big)|_{f_0} =2d_{00}(\lambda_1\lambda_2)|_{f_0}.
\end{align*}
This proves that the normal derivative is not determined by the shared vertex Hermite data.
\end{proof}
The non-conformity in $H^2(\Omega)$ causes severe difficulties in solving the biharmonic problem. In particular, the corresponding non-conforming method that replaces the Laplacian with the piecewise Laplacian diverges for certain meshes \cite[Thm.~4.6]{MingXu07}. Convergent alternatives, such as the TQC16 element \cite{MingXu07}, use the space $Z_3^0(T)$ but modify the definition of the discrete Laplacian/Hessian.
An alternative idea is to add new basis functions and thus allow for more degrees of freedom, which finally leads to $H^2(\Omega)$ conformity. However, if the space $Z^0_3(T)$ is enriched by polynomials, this increases the polynomial degrees of the trace spaces, thus requiring further degrees of freedom leading to more enrichments. This strategy leads to \v{Z}en\'{\i}\v{s}ek's polynomial element \cite{Zenisek73} with 220 local degrees of freedom, see also \cite{HuLinWu24} for generalizations to $C^r$ conforming elements in any dimension. To obtain smaller finite element spaces, one thus has to leave polynomial spaces, as discussed in Section~\ref{sec:Intro}. The alternative idea of the singular Zienkiewicz triangle is the enrichment of the space $Z_2^0(T)$ with rational bubble functions \cite{Ciarlet2002,DieningStornTscherpel24}. In the following we extend this idea to the three-dimensional case.
\section{Enrichment of the polynomial Zienkiewicz space}
The preceding discussion shows that the Zienkiewicz tetrahedron has too few degrees of freedom for $H^2$-conformity. Similar to the 2D case \cite[Chap.~6.6.1]{Ciarlet2002}, we thus suggest in this section the enrichment of the polynomial space by rational bubbles, allowing us to define additional degrees of freedom. Throughout this section, let  $T= [v_0,v_1,v_2,v_3] \subset \mathbb{R}^3$ be a tetrahedron with vertices $v_0,\dots,v_3$ and associated barycentric coordinates $\lambda_0,\dots,\lambda_3$.
\subsection{A class of rational functions}\label{sec:RatBubbles}
This subsection discusses a class of rational functions for enriching the polynomial Zienkiewicz space. Define the vector $\lambda\coloneqq(\lambda_0,\lambda_1,\lambda_2,\lambda_3)$ of barycentric coordinates. Let $\mathcal E_4\coloneqq\lbrace ij\colon 0\leq i<j\leq 3\rbrace \subset \mathbb{N}_0^2$ be the set of unordered pairs of vertices which can be identified with the set of edges. For multi-indices $\alpha\in\mathbb N_0^4$ and $\beta\in\mathbb N_0^{\mathcal E_4}$, we set the rational polynomial
\begin{align}\label{eq:DefRatBubble}
R^\alpha_\beta(\lambda)\coloneqq\frac{\lambda^\alpha}{\prod_{ij\in\mathcal E_4}(\lambda_i+\lambda_j)^{\beta_{ij}}},\qquad\text{where}\qquad  \lambda^\alpha\coloneqq\prod_{i=0}^3\lambda_i^{\alpha_i}.
\end{align}
The following lemma shows that the set of functions $R^\alpha_\beta$ with $\alpha\in \mathbb{N}_0^4$ and $\beta\in\mathbb N_0^{\mathcal E_4}$ is closed under multiplication. 
\begin{lemma}[Multiplication]\label{lem:Multi}
If $\alpha,\sigma\in\mathbb N_0^4$ and $\beta,\tau\in\mathbb N_0^{\mathcal E_4}$, then
\begin{align*}
R^\alpha_\beta R^\sigma_\tau=R^{\alpha+\sigma}_{\beta+\tau}.
\end{align*}
\end{lemma}
\begin{proof}
The lemma follows by the definition of the rational bubbles in \eqref{eq:DefRatBubble}.
\end{proof}
We can differentiate within the class of rational bubbles. For this, let $\fre_i\in\mathbb N_0^4$ be the $i$-th canonical unit vector and let $\fre_{ij}\in\mathbb N_0^{\mathcal E_4}$ be the unit vector associated with the pair $ij$ in the sense that $(\fre_{ij})_{k\ell}$ equals zero for $ij \neq k\ell$ and one for $ij = k\ell$. 
\begin{lemma}[Differentiation]\label{lem:Differentiatoin}
Let $k=0,\dots,3$, $\alpha \in \mathbb{N}_0^4$, and $\beta \in \mathbb N_0^{\mathcal E_4}$. Differentiation with respect to the barycentric coordinate $\lambda_k$ gives 
\begin{align*}
\partial_{\lambda_k}R^\alpha_\beta=
\begin{cases}
\alpha_k R^{\alpha-\fre_k}_\beta-\displaystyle\sum_{\substack{ij\in\mathcal E_4\\ k\in\lbrace i,j\rbrace}}\beta_{ij}R^\alpha_{\beta+\fre_{ij}}&\text{if }\alpha_k > 0,\\[2mm]
-\displaystyle\sum_{\substack{ij\in\mathcal E_4\\ k\in\lbrace i,j\rbrace}}\beta_{ij}R^\alpha_{\beta+\fre_{ij}}&\text{if }\alpha_k = 0.
\end{cases}
\end{align*}
\end{lemma}
\begin{proof}
Let $k=0,\dots,3$, $\alpha \in \mathbb{N}_0^4$, and $\beta \in \mathbb N_0^{\mathcal E_4}$. We write $R^\alpha_\beta=\lambda^\alpha D_\beta$ with
\begin{align*}
D_\beta\coloneqq \prod_{ij\in\mathcal E_4}(\lambda_i+\lambda_j)^{-\beta_{ij}}.
\end{align*}
The derivative of the monomial part is
\begin{align*}
\partial_{\lambda_k}\lambda^\alpha= 
\begin{cases}
\alpha_k\lambda^{\alpha-\fre_k}&\text{if }\alpha_k>0,\\
0&\text{if }\alpha_k=0.
\end{cases}
\end{align*}
It remains to differentiate the denominator factor. Only those factors depending on $\lambda_k$ contribute. These are precisely the factors indexed by edges $ij\in\mathcal E_4$ with $k\in\lbrace i,j\rbrace$. Hence, a calculation reveals
\begin{align*}
\partial_{\lambda_k}D_\beta
&= -\sum_{\substack{ij\in\mathcal E_4\\ k\in\lbrace i,j\rbrace}}\beta_{ij}(\lambda_i+\lambda_j)^{-\beta_{ij}-1}\prod_{\substack{\ell m\in\mathcal E_4\\ \ell m\neq ij}}(\lambda_\ell+\lambda_m)^{-\beta_{\ell m}} = -\sum_{\substack{ij\in\mathcal E_4\\ k\in\lbrace i,j\rbrace}}\beta_{ij}D_{\beta+\fre_{ij}}.
\end{align*}
Using the product rule, we obtain
\begin{align*}
\partial_{\lambda_k}R^\alpha_\beta = (\partial_{\lambda_k}\lambda^\alpha)D_\beta+\lambda^\alpha\partial_{\lambda_k}D_\beta.
\end{align*}
If $\alpha_k>0$, the first term equals $\alpha_k R^{\alpha-\fre_k}_\beta$, while if $\alpha_k=0$, it vanishes. The second term equals in both cases
\begin{align*}
\lambda^\alpha\partial_{\lambda_k}D_\beta = -\sum_{\substack{ij\in\mathcal E_4\\ k\in\lbrace i,j\rbrace}}\beta_{ij}R^\alpha_{\beta+\fre_{ij}}.
\end{align*}
This proves the asserted formula.
\end{proof}
\begin{remark}[Cartesian derivatives]
Since the barycentric coordinates are affine functions of the Cartesian variable $x \in \mathbb{R}^d$, the derivatives with respect to barycentric coordinates determine the Cartesian derivatives by the chain rule, cf.~\cite[Eq.~3.3]{DieningStornTscherpel24}. For example, we obtain
\begin{align}\label{eq:GradientBary}
\nabla R^\alpha_\beta = \sum_{k=0}^3 \partial_{\lambda_k}R^\alpha_\beta\nabla\lambda_k
\qquad\text{for all $\alpha \in \mathbb{N}_0^4$ and $\beta \in \mathbb N_0^{\mathcal E_4}$}.
\end{align}
This provides an explicit symbolic procedure for evaluating the derivatives required in finite element implementations.
\end{remark}
A rational function $R^\alpha_\beta$ might fail to be integrable due to possible singularities at boundary subsimplices where at least two barycentric coordinates vanish. To characterize this phenomenon, we define for $S\subset\lbrace0,1,2,3\rbrace$ the numbers
\begin{align*}
A_S(\alpha)\coloneqq \sum_{i\in S}\alpha_i\qquad\text{and}\qquad B_S(\beta)\coloneqq \sum_{\substack{i,j\in S\\ i<j}}\beta_{ij}.
\end{align*}
The quantity $B_S(\beta)$ measures the order of the denominator singularity along the boundary subsimplex
\begin{align*}
F_S\coloneqq \bigcap_{i\in S} \lbrace \lambda_i = 0\rbrace.
\end{align*}
\begin{lemma}[Regularity criterion]\label{lem:Regularity}
Let $m\in\mathbb N_0$, $1\leq p\leq\infty$, $\alpha \in \mathbb{N}_0^4$, and $\beta \in \mathbb N_0^{\mathcal E_4}$. Then we have $R^\alpha_\beta\in W^{m,p}(T)$ if and only if there holds for every $S\subset\lbrace0,1,2,3\rbrace$ with $2\leq |S|\leq 3$ and $B_S(\beta)>0$ that 
\begin{align}\label{eq:CondIntegr}
\begin{aligned}
m-|S|/p & < A_S(\alpha)-B_S(\beta)\qquad\text{if }p<\infty,\\
m & \leq A_S(\alpha)-B_S(\beta)\qquad\text{if }p=\infty.
\end{aligned}
\end{align}
\end{lemma}
\begin{proof}
It suffices to inspect a neighborhood of each point in the relative interior of every boundary subsimplex $F_S$ with $S\subset\lbrace0,1,2,3\rbrace$, $2\leq |S|\leq3$, and $B_S(\beta)>0$.
Since all barycentric coordinates outside $S$ stay bounded away from zero in a sufficiently small neighborhood of the relative interior of $F_S$, the possible singular behavior of $R_\beta^\alpha$ is described by
\begin{align*}
\prod_{i\in S}\lambda_i^{\alpha_i}\prod_{\substack{i,j\in S\\ i<j}}(\lambda_i+\lambda_j)^{-\beta_{ij}}.
\end{align*}
Set $\rho\coloneqq \sum_{i\in S}\lambda_i$. Since $\lambda_i\geq0$ for all $i\in S$, we may write $\lambda_i=\rho\theta_i$ with $\theta_i\geq0$ and $\sum_{i\in S}\theta_i=1$. 
The variable $\rho$ measures, up to equivalence of norms, the distance to $F_S$, whereas $\theta=(\theta_i)_{i\in S}$ describes the direction of approach.
The singular factor equals
\begin{align*}
\prod_{i\in S}\lambda_i^{\alpha_i}\prod_{\substack{i,j\in S\\ i<j}}(\lambda_i+\lambda_j)^{-\beta_{ij}} = \rho^{A_S(\alpha)-B_S(\beta)}
\prod_{i\in S}\theta_i^{\alpha_i}\prod_{\substack{i,j\in S\\ i<j}}(\theta_i+\theta_j)^{-\beta_{ij}}.
\end{align*}
Let $k\coloneqq |S|$ and assume without loss of generality $S = \lbrace 1,\dots,k\rbrace$. We consider first the integrability of the singular factor itself. 
We consider first directions for which $\theta_i\geq c>0$ for all $i\in S$. In this case all factors $\theta_i+\theta_j$ are bounded away from zero, and therefore the angular factor satisfies
\begin{align*}
\prod_{i\in S}\theta_i^{\alpha_i}\prod_{\substack{i,j\in S\\ i<j}}(\theta_i+\theta_j)^{-\beta_{ij}} \eqsim 1.
\end{align*}
Hence, the singular behavior is determined entirely by the radial factor $\rho^{A_S(\alpha)-B_S(\beta)}$.
Moreover, the change of variables $(\rho,\theta_1,\dots,\theta_{k-1}) \mapsto (\lambda_1,\dots,\lambda_k)$ with $\theta_k = 1- \theta_1 - \dots - \theta_{k-1}$, gives, up to a constant factor, the volume element
\begin{align*}
\rho^{k-1}\,\mathrm d\rho\,\mathrm d\theta.
\end{align*}
Consequently, for $p<\infty$, the contribution of such a neighborhood to the $L^p$-norm of the singular factor is finite if and only if
\begin{align*}
\int_0^\varepsilon \rho^{p(A_S(\alpha)-B_S(\beta))+k-1}\,\mathrm d\rho<\infty.
\end{align*}
Since $\int_0^\varepsilon \rho^q\,\mathrm d\rho<\infty$ if and only if $q>-1$, this is equivalent to
\begin{align*}
0 < p(A_S(\alpha)-B_S(\beta))+|S|\quad \text{or, equivalently,}\quad -|S|/p<A_S(\alpha)-B_S(\beta).
\end{align*}
For $p=\infty$, the corresponding condition is
\begin{align*}
0\leq A_S(\alpha)-B_S(\beta).
\end{align*}
It remains to note that the angular factor may itself become singular when $\theta$ approaches the boundary of the simplex of angular variables. Such a singularity can only occur if $\theta_i+\theta_j\to0$ for some pair $i,j\in S$ with $\beta_{ij}>0$. Since all $\theta_i$ are nonnegative, this means $\theta_i,\theta_j\to0$ and therefore corresponds precisely to approaching a singular boundary subsimplex associated with a proper subset $S'\subsetneq S$. These singularities are controlled by imposing the same condition for every subset $S'$ with $B_{S'}(\beta)>0$. Hence the above radial condition, applied to all such subsets, characterizes the local $L^p$-integrability of $R_\beta^\alpha$.

We now consider derivatives.
We choose local affine coordinates $(y,z)$ near the relative interior of $F_S$, where $y=(y_i)_{i\in S}$ with $y_i=\lambda_i$ and $z$ denotes coordinates tangential to $F_S$. Set
\begin{align*}
G_S(y)\coloneqq \prod_{i\in S}y_i^{\alpha_i}\prod_{\substack{i,j\in S\\ i<j}}(y_i+y_j)^{-\beta_{ij}}\qquad\text{for all }y\in\mathbb R_{>0}^S.
\end{align*}
The singular factor $G_S$ is homogeneous of degree $A_S(\alpha)-B_S(\beta)$; that is,
\begin{align*}
G_S(ty)=t^{A_S(\alpha)-B_S(\beta)}G_S(y)\qquad\text{for all }t>0\text{ and }y\in\mathbb R_{>0}^S.
\end{align*}
For a multi-index $\gamma=(\gamma_i)_{i\in S}\in\mathbb N_0^S$, differentiation with $\partial_y^\gamma\coloneqq\prod_{i\in S}\partial_{y_i}^{\gamma_i}$ and $|\gamma|\coloneqq\sum_{i\in S}\gamma_i$ yields
\begin{align*}
\partial_y^\gamma G_S(ty)=t^{A_S(\alpha)-B_S(\beta)-|\gamma|}\partial_y^\gamma G_S(y).
\end{align*}
Hence every transverse derivative of order $|\gamma|$ is homogeneous of degree $A_S(\alpha)-B_S(\beta)-|\gamma|$.
Up to a smooth factor with bounded derivatives, $R_\beta^\alpha$ is given by $G_S(y)$ in these coordinates. By the product rule, a derivative of total order at most $m$ is therefore a sum of terms in which at most $m$ derivatives act on $G_S$. Derivatives with respect to the tangential variables $z$ act only on the smooth factor and do not increase the singularity. Consequently, the strongest possible radial singularity occurs when $m$ transverse derivatives act on $G_S$ and is of order
\begin{align*}
\rho^{A_S(\alpha)-B_S(\beta)-m}.
\end{align*}
Repeating the preceding radial integration argument, all derivatives of order at most $m$ belong to $L^p$ for $p<\infty$ provided
\begin{align*}
p(A_S(\alpha)-B_S(\beta)-m)+|S|>0 \quad \text{or, equivalently,}\quad  m-|S|/p<A_S(\alpha)-B_S(\beta).
\end{align*}
For $p=\infty$, they are bounded provided
\begin{align*}
m\leq A_S(\alpha)-B_S(\beta).
\end{align*}
Since these conditions are imposed for every $S\subset\lbrace 0,1,2,3 \rbrace$ with $2\leq|S|\leq3$ and $B_S(\beta)>0$, they also control the possible singularities arising when the angular variable approaches the boundary of its simplex. This proves the asserted sufficient conditions.

Conversely, suppose that the condition in \eqref{eq:CondIntegr} fails for some $p\in [1,\infty)$ or $p=\infty$. Since $G_S$ is not a polynomial, there exists a multi-index $\gamma\in\mathbb N_0^S$ with $|\gamma|=m$ such that $\partial_y^\gamma G_S\not\equiv0$. By homogeneity,
\begin{align*}
\partial_y^\gamma G_S(\rho\theta)=\rho^{A_S(\alpha)-B_S(\beta)-m}\partial_y^\gamma G_S(\theta).
\end{align*}
Hence there exists an angular neighborhood on which $|\partial_y^\gamma G_S(\theta)|$ is bounded from below by a positive constant. Since the remaining factor in $R_\beta^\alpha$ is smooth and strictly positive, the corresponding derivative of $R_\beta^\alpha$ has the same leading radial behavior. If $p<\infty$ and $A_S(\alpha)-B_S(\beta)\leq m-|S|/p$, its $p$-th power is not integrable near $F_S$. If $p=\infty$ and $A_S(\alpha)-B_S(\beta)<m$, the derivative is unbounded near $F_S$. Thus $R_\beta^\alpha\notin W^{m,p}(T)$ whenever one of the stated conditions fails. This proves necessity and therefore the equivalence.
\end{proof}
\begin{remark}[Higher dimensions]
The proof of Lemma~\ref{lem:Regularity} is dimension-independent and extends verbatim to simplices in arbitrary spatial dimension. 
\end{remark}

\subsection{Singular edge bubbles}
We enrich the polynomial Zienkiewicz space by specific rational bubbles associated to edges. 
Let $e_{\ell m} = [v_\ell,v_m]$ with $\ell m \in \mathcal{E}_4$ be an edge of $T$ and let $n\in I\setminus\lbrace\ell,m\rbrace$ with $I\coloneqq\lbrace0,1,2,3\rbrace$. 
We set the edge bubble $b_{\ell m}\coloneqq \lambda_\ell\lambda_m$. Moreover, we denote by $i=i(\ell,m,n)$ the remaining index, so that
\begin{align*}
  \lbrace i,\ell,m,n\rbrace=I .
\end{align*}
We define the singular edge bubble
\begin{align}\label{eq:DefEellm}
  E_{\ell m}^n\coloneqq \frac{\lambda_n\lambda_\ell^2\lambda_m^2}{(\lambda_n+\lambda_\ell)(\lambda_n+\lambda_m)} = R_{\fre_{n\ell} + \fre_{nm}}^{\fre_n+2\fre_\ell+2\fre_m}.
\end{align}
Every edge carries two singular edge bubbles, one for each vertex not contained in the edge.
\begin{lemma}[Properties of the singular edge bubbles]
\label{lem:edge-bubble-properties}
Let $0\leq \ell<m\leq 3$, let $n\in I\setminus\lbrace\ell,m\rbrace$, and let $i=i(\ell,m,n)$ be such that $\lbrace i,\ell,m,n\rbrace=I$. Then $E_{\ell m}^n\in W^{2,\infty}(T)$. Moreover, we have
\begin{align*}
  E_{\ell m}^n|_{f_j}= \begin{cases}
 0&\text{for }j\in\lbrace\ell,m,n\rbrace,\\
 \frac{\lambda_n\lambda_\ell^2\lambda_m^2}{(\lambda_n+\lambda_\ell)(\lambda_n+\lambda_m)}\big|_{f_i} &\text{else}.   
  \end{cases}
\end{align*}
All scalar edge traces vanish and the edge-gradient traces satisfy
\begin{align*}
  \nabla E_{\ell m}^n|_{e}= \begin{cases}
b_{\ell m}\nabla\lambda_n&\text{for the edge } e = e_{\ell m},\\
0&\text{for all edges }e\neq e_{\ell m}.
\end{cases}  
\end{align*}
\end{lemma}
\begin{proof}
The function $E_{\ell m}^n = R_\beta^\alpha$ with $\beta = \fre_{n\ell} + \fre_{nm}$ and $\alpha = \fre_n + 2\fre_\ell + 2\fre_m$ belongs to the class of rational bubbles introduced in Section~\ref{sec:RatBubbles}.  
Hence the criterion in Lemma~\ref{lem:Regularity} applies for $m=2$ and $p=\infty$, and we obtain
\begin{align*}
E_{\ell m}^n\in W^{2,\infty}(T).
\end{align*}
The trace statements follow directly from the definition of $E_{\ell m}^n$. On the edge $e_{\ell m}$, we have $\lambda_n=0$ and $\lambda_i=0$. Differentiation in the $\lambda_n$ direction gives
\begin{align*}
  \partial_{\lambda_n}E_{\ell m}^n|_{e_{\ell m}}=\frac{\lambda_\ell^2\lambda_m^2}{\lambda_\ell\lambda_m}=b_{\ell m}.
\end{align*}
All other first derivatives vanish on $e_{\ell m}$. On every other edge, at least one of the factors $\lambda_\ell^2$ or $\lambda_m^2$ vanishes to second order, and the asserted gradient trace follows by \eqref{eq:GradientBary}.
\end{proof}
The lemma shows that the vertex Hermite data \eqref{eq:DOFShermite} of the singular edge bubbles vanish. Moreover, the edge bubbles generate precisely the missing quadratic part of the gradient trace on each edge. More precisely, let $e_{\ell m}$ be an edge and let
\begin{align*}
  I_{\ell m}\coloneqq I\setminus\lbrace \ell,m\rbrace=\lbrace n_0,n_1\rbrace .
\end{align*}
Then $\nabla\lambda_{n_0}$ and $\nabla\lambda_{n_1}$ span the two-dimensional space $e_{\ell m}^{\perp}\coloneqq \lbrace \nu\in\mathbb R^3\colon \nu\perp e_{\ell m}\rbrace$.
By Lemma~\ref{lem:edge-bubble-properties}, we have
\begin{align*}
  \nabla E_{\ell m}^{n}|_{e_{\ell m}}=b_{\ell m}\nabla\lambda_n\qquad\text{for }n\in I_{\ell m}.
\end{align*}
This yields
\begin{align*}
  \nabla\operatorname{span}\lbrace E_{\ell m}^{n}\colon n\in I_{\ell m}\rbrace\big|_{e_{\ell m}}=\operatorname{span}\lbrace b_{\ell m}\gamma\colon \gamma\in e_{\ell m}^{\perp}\rbrace .
\end{align*}
This gives two additional linearly independent edge-gradient components on every edge. We add these functions to the polynomial Zienkiewicz space. More precisely, with $J_n \coloneqq I \setminus \lbrace n \rbrace$ we set the space
\begin{align*}
 & \mathcal B_{\operatorname{edge}}(T)\coloneqq \operatorname{span}\lbrace E_{\ell m}^n\colon 0\leq \ell<m\leq 3,\ n\in I\setminus\lbrace\ell,m\rbrace\rbrace \\
  &= \operatorname{span}\lbrace R^\alpha_\beta \colon \alpha = \fre_n + 2 \fre_\ell + 2 \fre_m \text{ and } \beta = \fre_{n\ell} + \fre_{nm} \text{ with } n \in I, \ell,m \in J_n, \ell < m \rbrace.
\end{align*}
We define the edge-bubble-enriched space
\begin{align*}
  Z_3^e(T)\coloneqq Z_3^0(T)+\mathcal B_{\operatorname{edge}}(T)
\quad \text{with}\quad
  \dim Z_3^e(T)=\dim Z_3^0(T)+12=28.
\end{align*}
To set new degrees of freedom, we fix for every edge $e_{\ell m}$ of $T$ two linearly independent vectors
\begin{align*}
  \gamma_{\ell m}^0,\gamma_{\ell m}^1\in e_{\ell m}^{\perp}.
\end{align*}
This choice is made globally on the mesh $\tria$. Hence, if two tetrahedra share the same edge $e_{\ell m}$, then the same vectors $\gamma_{\ell m}^0,\gamma_{\ell m}^1$ are used from both sides.
For $z\in Z_3^e(T)$, the edge degrees of freedom are defined by
\begin{align}\label{eq:EdgeDofs}
  \mathcal E_{\ell m}^{j}(z)\coloneqq 4\,\nabla z(\operatorname{mid}(e_{\ell m}))\cdot\gamma_{\ell m}^{j}\qquad\text{for }j=0,1 \text{ and }0\leq \ell <m \leq 3.
\end{align}
If $I_{\ell m}=\lbrace n_0,n_1\rbrace$, then
\begin{align*}
  \mathcal E_{\ell m}^{j}(E_{\ell m}^{n_a})=\gamma_{\ell m}^{j}\cdot\nabla\lambda_{n_a}\qquad\text{for }j,a\in\lbrace 0,1 \rbrace.
\end{align*}
Moreover, we have for all $0 \leq \ell'<m'\leq 3$ the property
\begin{align*}
  \mathcal E_{\ell m}^{j}(E_{\ell' m'}^{n})=0\qquad\text{for }e_{\ell' m'}\neq e_{\ell m}.
\end{align*}
The degrees of freedom for the edge-enriched space $Z_3^e(T)$ are the vertex Hermite data \eqref{eq:DOFShermite} together with the edge-gradient degrees of freedom in \eqref{eq:EdgeDofs}.
There are $16+12=28$ degrees of freedom, matching the dimension of $Z_3^e(T)$.

\begin{lemma}[Unisolvence of the edge-enriched space]
\label{lem:unisolvence-edge-enriched}
The vertex Hermite degrees of freedom in \eqref{eq:DOFShermite} together with the edge-gradient degrees of freedom in \eqref{eq:EdgeDofs} are unisolvent for $Z_3^e(T)$.
\end{lemma}
\begin{proof}
Let $z\in Z_3^e(T)$ vanish in all these degrees of freedom. Since the edge bubbles have zero vertex Hermite data, the polynomial part $z_0\in Z_3^0(T)$ of $z$ has zero vertex Hermite data. By Lemma~\ref{lem:unisolvence}, we obtain $z_0=0$. Hence
\begin{align*}
z=\sum_{0\leq \ell<m\leq3}\sum_{n\in I_{\ell m}} c_{\ell m}^{n}E_{\ell m}^{n}\qquad\text{with coefficients } c_{\ell m}^{n} \in \mathbb{R}.
\end{align*}
Lemma~\ref{lem:Differentiatoin}  and \eqref{eq:GradientBary} give for an edge $e_{\ell m}$ and $I_{\ell m}=\lbrace n_0,n_1\rbrace$
\begin{align*}
\nabla z|_{e_{\ell m}}=b_{\ell m}\sum_{a=0}^1 c_{\ell m}^{n_a}\nabla\lambda_{n_a}.
\end{align*}
We set the vector $\xi_{\ell m}\coloneqq\sum_{a=0}^1 c_{\ell m}^{n_a}\nabla\lambda_{n_a}$.
Since the edge degrees of freedom vanish, we obtain
\begin{align*}
0=\mathcal E_{\ell m}^{j}(z)=\sum_{a=0}^1 c_{\ell m}^{n_a}\gamma_{\ell m}^{j}\cdot\nabla\lambda_{n_a} = \xi_{\ell m} \cdot \gamma_{\ell m}^{j} \qquad\text{for }j=0,1.
\end{align*}
Since $\lambda_{n_0}$ and $\lambda_{n_1}$ vanish on $e_{\ell m}$, their gradients are orthogonal to the edge $e_{\ell m}$ and consequently $\xi_{\ell m}\in e_{\ell m}^{\perp}$. Moreover, the identity above shows that $\xi_{\ell m}$ is orthogonal to both $\gamma_{\ell m}^0$ and $\gamma_{\ell m}^1$. Since $\gamma_{\ell m}^0,\gamma_{\ell m}^1$ form a basis of $e_{\ell m}^{\perp}$, it follows that $\xi_{\ell m}=0$. The vectors $\nabla\lambda_{n_0}$ and $\nabla\lambda_{n_1}$ are linearly independent, and therefore $c_{\ell m}^{n_0}=c_{\ell m}^{n_1}=0$.
Since the edge $e_{\ell m}$ was arbitrary, all coefficients vanish and therefore $z=0$.
\end{proof}
Let $f_i$ be a fixed facet with normal unit vector $\nu_i\in \mathbb{R}^3$. If $e_{\ell m}\subset f_i$, then the remaining vertex of $f_i$ is denoted by
\begin{align*}
  \lbrace n\rbrace =  I \setminus \lbrace i,\ell,m\rbrace.
\end{align*}
The corresponding two-dimensional rational edge bubble on $f_i$ is
\begin{align*}
  E_{\ell m}^{n,i}\coloneqq \frac{\mu_n\mu_\ell^2\mu_m^2}{(\mu_n+\mu_\ell)(\mu_n+\mu_m)} = E_{\ell m}^n|_{f_i} .
\end{align*}
The scalar trace of the edge-enriched space on $f_i$ is $Z_3^e(T)|_{f_i}=Z_2^s(f_i)$ reads with $Z_2^0(f_i)$ defined in Lemma~\ref{lem:FacetTracePoly}
\begin{align*}
  Z_2^s(f_i)\coloneqq Z_2^0(f_i)+\operatorname{span}\lbrace E_{\ell m}^{n,i}\colon e_{\ell m}\subset f_i\text{ and }  \lbrace n \rbrace =I \setminus\lbrace i,\ell,m\rbrace\rbrace.
\end{align*}
Thus the scalar facet trace is the two-dimensional singular Zienkiewicz space on $f_i$, which shows continuity of the corresponding global finite element spaces.
To describe the normal traces created by the edge bubbles, we set
\begin{align*}
  \Phi_{\ell m,\ell}^{n,i}\coloneqq \partial_{\mu_\ell}E_{\ell m}^{n,i},\qquad \Phi_{\ell m,m}^{n,i}\coloneqq \partial_{\mu_m}E_{\ell m}^{n,i},\qquad \Phi_{\ell m,n}^{n,i}\coloneqq \partial_{\mu_n}E_{\ell m}^{n,i}-\mu_\ell\mu_m.
\end{align*}
Using Lemma~\ref{lem:Differentiatoin} (more precisely, its 2D version in \cite[Lem.~3.1]{DieningStornTscherpel24}), or a direct calculation, we obtain
\begin{align}\label{eq:sadfqq0}
\begin{aligned}
\Phi_{\ell m,\ell}^{n,i}
&=\partial_{\mu_\ell}E_{\ell m}^{n,i}
=\frac{\mu_n\mu_m^2}{\mu_n+\mu_m}\partial_{\mu_\ell}\left(\frac{\mu_\ell^2}{\mu_n+\mu_\ell}\right)\\
&=\frac{\mu_n\mu_m^2}{\mu_n+\mu_m}
\frac{2\mu_\ell(\mu_n+\mu_\ell)-\mu_\ell^2}{(\mu_n+\mu_\ell)^2} =\frac{\mu_n\mu_\ell\mu_m^2(2\mu_n+\mu_\ell)}{(\mu_n+\mu_\ell)^2(\mu_n+\mu_m)}.
\end{aligned}
\end{align}
Similar calculations reveal
\begin{align}\label{eq:sadfqq}
\begin{aligned}
\Phi_{\ell m,m}^{n,i} & = \frac{\mu_n\mu_\ell^2\mu_m(2\mu_n+\mu_m)}{(\mu_n+\mu_\ell)(\mu_n+\mu_m)^2},\\
\Phi_{\ell m,n}^{n,i} & = \frac{\mu_\ell^2\mu_m^2(\mu_\ell\mu_m-\mu_n^2)}{(\mu_n+\mu_\ell)^2(\mu_n+\mu_m)^2}-\mu_\ell\mu_m.
\end{aligned}
\end{align}
\begin{lemma}[Vanishing trace]\label{lem:VanishingTrace}
The trace of the functions $\Phi_{\ell m,\ell}^{n,i}$, $\Phi_{\ell m,m}^{n,i}$, and $\Phi_{\ell m,n}^{n,i}$ equals zero on the face boundary $\partial f_i$.
\end{lemma}
\begin{proof}
The boundary $\partial f_i$ is the union of the three edges given by $\mu_\ell=0$, $\mu_m=0$, and $\mu_n=0$. For the first two functions the claim follows directly from the explicit formulas, since $\Phi_{\ell m,\ell}^{n,i}$ and $\Phi_{\ell m,m}^{n,i}$  contain the factor $\mu_n\mu_\ell\mu_m$ in the numerator.
It remains to consider $\Phi_{\ell m,n}^{n,i}$. If $\mu_\ell=0$ or $\mu_m=0$, then both terms in \eqref{eq:sadfqq} vanish. On the remaining edge $e_{\ell m}$ we have $\mu_n=0$, which yields
\begin{align*}
  \Phi_{\ell m,n}^{n,i}|_{e_{\ell m}}
  = \frac{\mu_\ell^2\mu_m^2\mu_\ell\mu_m}{\mu_\ell^2\mu_m^2}\Big|_{e_{\ell m}} - \mu_\ell\mu_m|_{e_{\ell m}} = 0.
\end{align*}
Thus all three functions equal zero on the relative interiors of the boundary edges. The values at the vertices are understood by continuous extension, and therefore the traces vanish on all of $\partial f_i$.
\end{proof}
Set $d_{ji}\coloneqq \nabla\lambda_j\cdot\nu_i$ for all $i,j\in I$ and define the space
\begin{align*}
  Q_{\operatorname{edge}}(f_i)&\coloneqq \operatorname{span}\lbrace \Phi_{\ell m,\ell}^{n,i},\Phi_{\ell m,m}^{n,i},\Phi_{\ell m,n}^{n,i}\colon e_{\ell m}\subset f_i\text{ and } \lbrace n \rbrace =I \setminus\lbrace i,\ell,m\rbrace\rbrace.
\end{align*}

\begin{lemma}[Normal traces]\label{lem:normal-traces-edge-bubbles}
For every facet $f_i$ there holds
\begin{align*}
  \partial_{\nu_i}Z_3^e(T)|_{f_i}\subset N_2(f_i) \coloneqq \mathcal P_2(f_i) \oplus Q_{\operatorname{edge}}(f_i).
\end{align*}
More precisely, if $e_{\ell m}\subset f_i$ and $ \lbrace n \rbrace =I \setminus\lbrace i,\ell,m\rbrace$, then
\begin{align}\label{eq:IdentLem}
  \partial_{\nu_i}E_{\ell m}^n|_{f_i}=d_{\ell i}\Phi_{\ell m,\ell}^{n,i}+d_{mi}\Phi_{\ell m,m}^{n,i}+d_{ni}\Phi_{\ell m,n}^{n,i}+d_{ni}\mu_\ell\mu_m.
\end{align}
On the remaining facets there holds
\begin{align}
\label{eq:IdentLemOtherFaces}
\partial_{\nu_n}E_{\ell m}^n|_{f_n} = d_{nn}\mu_\ell\mu_m, \qquad \partial_{\nu_\ell}E_{\ell m}^n|_{f_\ell} = 0, \qquad \partial_{\nu_m}E_{\ell m}^n|_{f_m} = 0.
\end{align}
\end{lemma}
\begin{proof}
The chain rule \eqref{eq:GradientBary} gives
\begin{align*}
  \nabla E_{\ell m}^n = \partial_{\lambda_\ell}E_{\ell m}^n\nabla\lambda_\ell + \partial_{\lambda_m}E_{\ell m}^n\nabla\lambda_m + \partial_{\lambda_n}E_{\ell m}^n\nabla\lambda_n.
\end{align*}
Multiplying the restricting to the facet $f_i$ by $\nu_i$ thus yields
\begin{align*}
  \partial_{\nu_i}E_{\ell m}^n|_{f_i} = d_{\ell i}\partial_{\lambda_\ell}E_{\ell m}^n|_{f_i} + d_{mi}\partial_{\lambda_m}E_{\ell m}^n|_{f_i} +  d_{ni}\partial_{\lambda_n}E_{\ell m}^n|_{f_i}.
\end{align*}
On $f_i$ we have $\mu_j=\lambda_j|_{f_i}$ for all $j\in J_i$, and $E_{\ell m}^{n,i} = E_{\ell m}^n|_{f_i}$, resulting in
\begin{align*}
  \partial_{\lambda_j}E_{\ell m}^n|_{f_i} = \partial_{\mu_j}E_{\ell m}^{n,i}  \qquad\text{for all }j\in\lbrace\ell,m,n\rbrace.
\end{align*}
Combining these results with the identities defining the functions $\Phi_{\ell m,j}^{n,i}$ gives \eqref{eq:IdentLem}.
The last term in \eqref{eq:IdentLem} belongs to $\mathcal P_2(f_i)$, while the remaining terms belong to $Q_{\operatorname{edge}}(f_i)$. Combining this with the property that the polynomial Zienkiewicz part has normal trace in $\mathcal P_2(f_i)$, we obtain 
\begin{align*}
\partial_{\nu_i}Z_3^e(T)|_{f_i}\subset N_2(f_i) \coloneqq \mathcal P_2(f_i) + Q_{\operatorname{edge}}(f_i).
\end{align*}
The latter sum is direct, as any function in $Q_{\operatorname{edge}}(f_i)$ has vanishing trace on $\partial f_i$, see Lemma~\ref{lem:VanishingTrace}. Consequently, any $p\in \mathcal P_2(f_i)\cap Q_{\operatorname{edge}}(f_i)$ must satisfy $p|_{\partial f_i}=0$. However, any quadratic polynomial $p$ on a triangle that is zero on the boundary must equal $p=0$. 

It remains to consider the other facets. The definition of the singular edge bubble in \eqref{eq:DefEellm} yields that $E_{\ell m}^n$ vanishes quadratically in $\lambda_\ell$ on $f_\ell$ and quadratically in $\lambda_m$ on $f_m$. Consequently, we have
\begin{align*}
\nabla E_{\ell m}^n|_{f_\ell}=0 \qquad\text{and}\qquad \nabla E_{\ell m}^n|_{f_m}=0.
\end{align*}
On $f_n$ the function vanishes only linearly in $\lambda_n$. Moreover, we have
\begin{align*}
\partial_{\lambda_\ell}E_{\ell m}^n|_{f_n} = \partial_{\lambda_m}E_{\ell m}^n|_{f_n} = 0 \qquad\text{and}\qquad \partial_{\lambda_n}E_{\ell m}^n|_{f_n} = \mu_\ell\mu_m.
\end{align*}
This yields in combination with \eqref{eq:GradientBary} the identities in \eqref{eq:IdentLemOtherFaces}.
\end{proof}
Since the rational edge bubbles introduce non-polynomial normal traces on the faces $f_i$, see Lemma~\ref{lem:normal-traces-edge-bubbles}, the enrichment of the polynomial Zienkiewicz space by rational edge bubbles does not suffice to design a globally $H^2$ conforming finite element. To remedy this side-effect, we correct the rational part of the normal trace. 

\subsection{Rational normal trace correction}
Let $z\in Z_3^e(T)$. Since we have $\mathcal P_2(f_i)\cap Q_{\operatorname{edge}}(f_i)=\lbrace0\rbrace$ for every facet $f_i$, see Lemma~\ref{lem:normal-traces-edge-bubbles}, there is a unique decomposition
\begin{align}\label{eq:DecompZface}
  \partial_{\nu_i}z|_{f_i}=p_i(z)+q_i(z)\qquad \text{with } p_i(z)\in\mathcal P_2(f_i) \text{ and }q_i(z)\in Q_{\operatorname{edge}}(f_i).
\end{align}
We aim at removing the rational part $q_i(z)$ by subtracting a normal lift.
For this, we define for all $q \in Q_{\operatorname{edge}}(f_i)$ its canonical barycentric extension $\widetilde q$ to $T$  by replacing $\mu_j$ with $\lambda_j$ for all $j\in J_i$. Then we set the normal lift
\begin{align}\label{eq:TraceExtension}
L_i q\coloneqq \frac{\lambda_i}{d_{ii}}\Theta_i\widetilde q  \qquad\text{with}\qquad  \Theta_i\coloneqq \prod_{j\in J_i}\frac{\lambda_j}{\lambda_i+\lambda_j}\text{ and } d_{ii} \coloneqq \nabla \lambda_i \cdot \nu_i.
\end{align}
\begin{lemma}[Regularity]
\label{lem:regularity-normal-lifting}
Let $i\in I$ and let $q\in Q_{\operatorname{edge}}(f_i)$. Then
\begin{align*}
  L_iq\in W^{2,\infty}(T).
\end{align*}
\end{lemma}
\begin{proof}
It suffices to prove the assertion for the generators of $Q_{\operatorname{edge}}(f_i)$. Let $e_{\ell m}\subset f_i$ and let $\lbrace n\rbrace=I\setminus\lbrace i,\ell,m\rbrace$. We use the regularity criterion of Lemma~\ref{lem:Regularity}, which says $W^{2,\infty}(T)$ regularity of a rational function $R_\beta^\alpha$ requires 
\begin{align*}
 2 \leq  A_S(\alpha)-B_S(\beta)\qquad \text{for all $S\subset I$ with $2\leq |S|\leq 3$ and $B_S(\beta)>0$}.
\end{align*}
We first consider
\begin{align*}
  L_i\Phi_{\ell m,\ell}^{n,i} = \frac{\lambda_i}{d_{ii}}   \frac{\lambda_n\lambda_\ell\lambda_m^2(2\lambda_n+\lambda_\ell)}{(\lambda_n+\lambda_\ell)^2(\lambda_n+\lambda_m)}\prod_{j\in J_i}\frac{\lambda_j}{\lambda_i+\lambda_j}.
\end{align*}
This function is a linear combination of two functions $R_\beta^\alpha$ with
\begin{align*}
  \beta = \fre_{i\ell}+\fre_{im}+\fre_{in}+2\fre_{n\ell}+\fre_{nm}
\quad \text{and}\quad 
  \alpha= \begin{cases}
\fre_i+2\fre_\ell+3\fre_m+3\fre_n&\text{or}\\
\fre_i+3\fre_\ell+3\fre_m+2\fre_n.
\end{cases}  
\end{align*}
For both choices of $\alpha$ we have  for all relevant subsets $S \subset \lbrace i,\ell,m,n \rbrace$ that
\begin{align*}
2 \leq  A_S(\alpha)-B_S(\beta).
\end{align*}
Hence $L_i\Phi_{\ell m,\ell}^{n,i}\in W^{2,\infty}(T)$. The argument for $L_i\Phi_{\ell m,m}^{n,i}$ is identical.
It remains to consider $L_i\Phi_{\ell m,n}^{n,i}$. We use the identity
\begin{align*}
  \Phi_{\ell m,n}^{n,i} = \frac{\mu_\ell^2\mu_m^2(\mu_\ell\mu_m-\mu_n^2)}{(\mu_n+\mu_\ell)^2(\mu_n+\mu_m)^2}-\mu_\ell\mu_m.
\end{align*}
After passing to the canonical barycentric extension and putting the two terms over a common denominator, the numerator factorizes as
\begin{align*}
  &\lambda_\ell^2\lambda_m^2(\lambda_\ell\lambda_m-\lambda_n^2) -\lambda_\ell\lambda_m(\lambda_n+\lambda_\ell)^2(\lambda_n+\lambda_m)^2 \\
  & = - \lambda_\ell\lambda_m \big( (\lambda_n+\lambda_\ell)^2(\lambda_n+\lambda_m)^2 - \lambda_\ell\lambda_m(\lambda_\ell\lambda_m-\lambda_n^2)\big)\\ 
  & =  -\lambda_\ell\lambda_m\lambda_n \big( 2\lambda_\ell^2\lambda_m +\lambda_\ell^2\lambda_n +2\lambda_\ell\lambda_m^2 +5\lambda_\ell\lambda_m\lambda_n  +2\lambda_\ell\lambda_n^2 +\lambda_m^2\lambda_n +2\lambda_m\lambda_n^2+\lambda_n^3\big).
\end{align*}
Consequently, $L_i\Phi_{\ell m,n}^{n,i}$ is a linear combination of functions $R_\beta^\alpha$ which, with
\begin{align*}
\tau \in \lbrace 2\fre_\ell + \fre_m, 2\fre_\ell + \fre_n,\fre_\ell + 2\fre_m,\fre_\ell + 2\fre_n, 2\fre_m+\fre_n,\fre_\ell+\fre_m+\fre_n, \fre_m + 2 \fre_n, 3 \fre_n\rbrace,
\end{align*}
read
\begin{align*}
  \beta = \fre_{i\ell}+\fre_{im}+\fre_{in}+2\fre_{n\ell}+2\fre_{nm}\quad\text{and}\quad  \alpha=\fre_i+2\fre_\ell+2\fre_m+2\fre_n+\tau.
\end{align*}
For all these choices of $\alpha$ one has
\begin{align*}
  2\leq A_S(\alpha)-B_S(\beta)\qquad\text{for every $S\subset I$ with $2\leq |S|\leq 3$ and $B_S(\beta)>0$}.
\end{align*}
Lemma~\ref{lem:Regularity} therefore yields $L_i\Phi_{\ell m,n}^{n,i}\in W^{2,\infty}(T)$.
Since any $q\in Q_{\operatorname{edge}}(f_i)$ is a finite linear combination of the above generators, the claim follows.
\end{proof}
Our construction allows us to remove singular normal traces, as depicted in the following lemma.
\begin{lemma}[Normal lifting]
\label{lem:normal-face-bubbles}
Let $i\in I$ and $q\in Q_{\operatorname{edge}}(f_i)$. Then we have
\begin{align*} 
L_iq|_{\partial T}=0\qquad\text{and}\qquad
  \partial_{\nu_i}L_iq|_{f_j}= \begin{cases} q &\text{for }i  = j,\\
  0& \text{for }j\neq i.
  \end{cases}
\end{align*}
Moreover, $\nabla L_iq$ vanishes on every edge of $T$.
\end{lemma}

\begin{proof}
Recall the definition of $L_iq$ in \eqref{eq:TraceExtension}.
Its scalar trace is zero on $f_i$ because of the factor $\lambda_i$. If $j\neq i$, then $\Theta_i$ contains the factor $\lambda_j/(\lambda_i+\lambda_j)$, and hence its scalar trace is zero on $f_j$. Thus, we have $L_iq|_{\partial T}=0$.
We next compute the normal traces. On $f_i$ we have $\lambda_i=0$ and $\Theta_i=1$, resulting in
\begin{align*}
  \partial_{\nu_i}L_iq|_{f_i} = \frac{1}{d_{ii}} \partial_{\nu_i}\big(\lambda_i\Theta_i\widetilde q\big)|_{f_i} = \frac{1}{d_{ii}} (\nabla\lambda_i\cdot\nu_i)\Theta_i\widetilde q|_{f_i}
  = q.
\end{align*}
Now let $j\neq i$. Rearranging terms yields
\begin{align*}
  L_iq = \frac{\lambda_i}{d_{ii}}   \frac{\lambda_j}{\lambda_i+\lambda_j}   \Theta_{ij}\widetilde q \qquad\text{with }  \Theta_{ij} \coloneqq  \prod_{k\in J_i\setminus\lbrace j\rbrace}\frac{\lambda_k}{\lambda_i+\lambda_k}.
\end{align*}
Since $\lambda_j/(\lambda_i + \lambda_j)$ and $\lambda_j$ vanish on $f_j$, the product and quotient rule gives
\begin{align*}
\partial_{\nu_j}L_iq|_{f_j} &= \left(\frac{\lambda_i}{d_{ii}}\Theta_{ij}\widetilde q\, \partial_{\nu_j}\frac{\lambda_j}{\lambda_i+\lambda_j}\right)\bigg|_{f_j} = \left(\frac{\lambda_i}{d_{ii}}\Theta_{ij}\widetilde q\, \frac{\lambda_i\partial_{\nu_j}\lambda_j-\lambda_j\partial_{\nu_j}\lambda_i}{(\lambda_i+\lambda_j)^2}\right)\bigg|_{f_j}\\
& =  \frac{1}{d_{ii}}\Theta_{ij}\widetilde q|_{f_j}\partial_{\nu_j}\lambda_j.
\end{align*}
Every function $q\in Q_{\operatorname{edge}}(f_i)$ is homogeneous of degree two in the barycentric coordinates $(\mu_k)_{k\in J_i}$. Hence, the vanishing of $q$ on the edge $\mu_j=0$ implies that its canonical barycentric extension satisfies $\widetilde q|_{f_j}=0$.
Consequently, we obtain $\partial_{\nu_j}L_iq|_{f_j}=0$ for every $j\neq i$. 
Since, the function $L_iq$ equals zero at all edges of $T$, its directional derivative along these edges equals zero. Combining this with the first part of the lemma shows that the gradient at all edges is zero.
\end{proof}

With this operator we define for any $z\in Z_3^e(T)$ with normal derivative $ \partial_{\nu_i} z|_{f_i} = p_i(z) + q_i(z)$ as in  \eqref{eq:DecompZface} the correction operator
\begin{align}\label{eq:CorrectionOperator}
 \mathcal C_T z\coloneqq \sum_{i=0}^3 L_iq_i(z) \qquad\text{and}\qquad   \mathcal R_T z\coloneqq z-\mathcal C_Tz.
\end{align}

\subsection{The singular Zienkiewicz tetrahedron}\label{subbsec:SingZienk}
Using the corrector $\mathcal{R}_T$ defined in \eqref{eq:CorrectionOperator}, we set the singular Zienkiewicz space
\begin{align*}
  Z_3^{\operatorname{sing}}(T)\coloneqq\mathcal R_T Z_3^e(T).
\end{align*}
The following lemma introduces a more direct characterization.
\begin{lemma}[Direct characterization]
\label{lem:direct-characterization-singular-zienkiewicz}
For $0\leq \ell<m\leq3$ and $n\notin\lbrace\ell,m\rbrace$, let $i$ denote the unique index in $\lbrace0,1,2,3 \rbrace \setminus \lbrace \ell,m,n \rbrace$. Using the canonical barycentric extensions of the functions in \eqref{eq:sadfqq0}, we define
\begin{align*}
\widehat E_{\ell m}^n\coloneqq E_{\ell m}^n-\frac{\lambda_i}{d_{ii}}\Theta_i\left(d_{\ell i} \widetilde \Phi_{\ell m,\ell}^{n,i}+d_{mi}\widetilde \Phi_{\ell m,m}^{n,i}+d_{ni} \widetilde \Phi_{\ell m,n}^{n,i}\right).
\end{align*}
Then the singular Zienkiewicz space admits the representation
\begin{align*}
Z_3^{\operatorname{sing}}(T)=Z_3^0(T)+\operatorname{span}\lbrace\widehat E_{\ell m}^n\colon 0\leq\ell<m\leq3,\ n\notin\lbrace\ell,m\rbrace\rbrace.
\end{align*}
\end{lemma}
\begin{proof}
By the definition of the edge-enriched space,
\begin{align*}
Z_3^e(T)=Z_3^0(T)+\operatorname{span}\lbrace E_{\ell m}^n\colon 0\leq\ell<m\leq3,\ n\notin\lbrace\ell,m\rbrace\rbrace.
\end{align*}
For $e_{\ell m}\subset f_i$ and $\lbrace n\rbrace=\lbrace 0,1,2,3 \rbrace\setminus\lbrace i,\ell,m\rbrace$, the normal-trace identity \eqref{eq:IdentLem} yields the rational component of the normal trace
\begin{align*}
q_i(E_{\ell m}^n)=d_{\ell i}\Phi_{\ell m,\ell}^{n,i}+d_{mi}\Phi_{\ell m,m}^{n,i}+d_{ni}\Phi_{\ell m,n}^{n,i}.
\end{align*}
Hence, by the definition of the normal lift in \eqref{eq:TraceExtension}, we obtain
\begin{align*}
L_iq_i(E_{\ell m}^n)=\frac{\lambda_i}{d_{ii}}\Theta_i\left(d_{\ell i}\widetilde \Phi_{\ell m,\ell}^{n,i}+d_{mi}\widetilde \Phi_{\ell m,m}^{n,i}+d_{ni}\widetilde \Phi_{\ell m,n}^{n,i}\right).
\end{align*}
Consequently, it holds that $\mathcal R_TE_{\ell m}^n=E_{\ell m}^n-L_iq_i(E_{\ell m}^n)=\widehat E_{\ell m}^n$.
The corrector $\mathcal R_T$ is linear and acts as the identity on $Z_3^0(T)$, since functions in $Z_3^0(T)$ have polynomial normal traces and hence vanishing rational normal-trace components. Therefore, the assertion follows by 
\begin{align*}
\mathcal R_TZ_3^e(T)&=Z_3^0(T)+\operatorname{span}\lbrace\widehat E_{\ell m}^n\colon 0\leq\ell<m\leq3,\ n\notin\lbrace\ell,m\rbrace\rbrace.\qedhere
\end{align*}
\end{proof}
The following lemma shows that $Z_3^{\operatorname{sing}}(T)$ is a finite element with the vertex  Hermite degrees of freedom \eqref{eq:DOFShermite} and the edge-gradient degrees of freedom \eqref{eq:EdgeDofs}.
\begin{lemma}[Singular Zienkiewicz tetrahedron]
\label{lem:unisolvence-singular-reduced}
The vertex Hermite degrees of freedom \eqref{eq:DOFShermite} together with the edge-gradient degrees of freedom \eqref{eq:EdgeDofs} are unisolvent for $Z_3^{\operatorname{sing}}(T)$ and we have 
\begin{align*}
\dim Z_3^{\operatorname{sing}}(T)=28.
\end{align*}
\end{lemma}
\begin{proof}
We first observe that the correction operator does not change the degrees of freedom. Let $z\in Z_3^e(T)$ and write
\begin{align*}
 \mathcal R_T z = z - \mathcal{C}_Tz \qquad\text{with}\qquad  \mathcal C_Tz=\sum_{i=0}^3 L_iq_i(z).
\end{align*}
By Lemma~\ref{lem:normal-face-bubbles}, each lifting $L_iq_i(z)$ has zero scalar trace on $\partial T$ and its gradient vanishes on every edge of $T$. In particular, we obtain
\begin{align*}
  (\mathcal R_Tz)(v_r)=z(v_r)\qquad\text{and}\qquad \nabla(\mathcal R_Tz)(v_r)=\nabla z(v_r)
\end{align*}
for all vertices $v_r$ as well as for all edge-gradient degrees of freedom
\begin{align*}
  \mathcal E_{\ell m}^j(\mathcal R_Tz)=\mathcal E_{\ell m}^j(z).
\end{align*}
Consequently, the degrees of freedom uniquely determine $z$ and consequently $\mathcal{R}_T z$, which yields unisolvence of $Z_3^{\operatorname{sing}}(T)$ and $\dim Z_3^{\operatorname{sing}}(T) = \dim Z_3^e(T)=28$.
%
\end{proof}
Let $\tria$ be a conforming tetrahedral triangulation of $\Omega$. For every mesh edge $e_{\ell m}$, fix the two linearly independent vectors $\gamma_{\ell m}^0,\gamma_{\ell m}^1\in e_{\ell m}^{\perp}$ globally. We define the global singular Zienkiewicz space by
\begin{align*}
Z_{3}^{\operatorname{sing}}(\tria) \coloneqq \lbrace z\in L^2(\Omega)&\colon z|_T\in Z_3^{\operatorname{sing}}(T)\text{ for all }T\in\tria,\ \text{and all vertex Hermite}\\
 &\ \ \text{and edge-gradient degrees of freedom are single-valued}\rbrace .
\end{align*}
\begin{theorem}[Global $H^2$ conformity]
\label{thm:global-h2-conformity-singular}
The global finite element space satisfies
\begin{align*}
Z_{3}^{\operatorname{sing}}(\tria)\subset H^2(\Omega).
\end{align*}
\end{theorem}

\begin{proof}
Let $z\in Z_{3}^{\operatorname{sing}}(\tria)$ and let $f=T^+\cap T^-$ be an interior facet. We prove that the scalar trace and the full gradient trace coincide on the face $f$.

\textit{Step 1 (Scalar trace).}
The correction operator does not change scalar traces on $\partial T$. Hence, for each element $T$, the trace of $Z_3^{\operatorname{sing}}(T)$ on $f$ equals the trace of $Z_3^e(T)$ on $f$, namely the two-dimensional singular Zienkiewicz trace space $Z_2^s(f)$. By the trace unisolvence on $f$, this scalar trace is determined by the vertex Hermite data at the three vertices of $f$ and by the edge-gradient degrees of freedom on the three edges of $f$. These degrees of freedom are single-valued by definition of the global space, which yields $z|_{T^+}=z|_{T^-}$ on $f$.

\textit{Step 2 (Gradient trace).}
Since continuity yields continuity of the tangential gradients on faces $f$, it remains to prove continuity of the normal derivative. We fix a unit normal vector $\nu_f$ on $f$. By construction of the correction operator, the rational part of the normal trace has been removed. Hence
\begin{align*}
  \partial_{\nu_f}(z|_{T^+})|_f\in\mathcal P_2(f)
  \qquad\text{and}\qquad
  \partial_{\nu_f}(z|_{T^-})|_f\in\mathcal P_2(f).
\end{align*}
The values of these two quadratic polynomials at the vertices of $f$ are determined by the vertex gradients
\begin{align*}
  \partial_{\nu_f}(z|_{T^\pm})(v)=\nabla z(v)\cdot\nu_f
  \qquad\text{for every vertex }v\text{ of }f.
\end{align*}
These values agree from both sides because the vertex Hermite data are single-valued.
Let now $e\subset f$ be an edge and let $\operatorname{mid}(e)$ denote its midpoint. Since $\nu_f\perp e$, we have $\nu_f\in e^\perp$. Thus $\nu_f$ is a linear combination of the globally fixed vectors $\gamma_e^0,\gamma_e^1$. Consequently, the edge-gradient degrees of freedom determine
\begin{align*}
  \nabla z(\operatorname{mid}(e))\cdot\nu_f .
\end{align*}
These values are single-valued by definition of the global space. Hence the two quadratic polynomials $\partial_{\nu_f}(z|_{T^+})|_f$ and $\partial_{\nu_f}(z|_{T^-})|_f$ agree at the three vertices and at the three edge midpoints of $f$. Since these six point values are unisolvent for $\mathcal P_2(f)$, we obtain $\partial_{\nu_f}(z|_{T^+})=\partial_{\nu_f}(z|_{T^-})$ on $f$.
This verifies global $H^2$ regularity.
\end{proof}

\subsection{The cubic singular Zienkiewicz tetrahedron} 
The singular space $Z_3^{\operatorname{sing}}(T)$ does not contain all cubic polynomials, reducing the order of convergence. To remedy this drawback we add the four missing cubic face bubbles. In particular, we define for every facet $f_i$ with $i \in \lbrace 0,\dots, 3\rbrace$ and $J_i \coloneqq \lbrace 0,\dots,3\rbrace \setminus \lbrace i \rbrace$ the face bubble
\begin{align*}
  B_i\coloneqq \prod_{j\in J_i}\lambda_j \qquad\text{and set}\qquad \mathcal B_{\operatorname{cub}}(T)\coloneqq \operatorname{span}\lbrace B_i\colon i=0,\dots,3\rbrace.
\end{align*}
For every $i\in\lbrace0,1,2,3\rbrace$, define the corrected cubic facet bubble
\begin{align*}
W_i \coloneqq B_i - \sum_{\substack{e_{\ell m}\subset f_i\\ \lbrace i,\ell,m,n\rbrace=\lbrace0,1,2,3\rbrace}} \mathcal{R}_T E_{\ell m}^{n}\quad\text{and}\quad \mathcal B^{\operatorname{cor}}_{\operatorname{cub}}(T)\coloneqq \operatorname{span}\lbrace W_i\colon i=0,\dots,3\rbrace .
\end{align*}
\begin{lemma}[Corrected bubble]\label{lem:CorrBubble}
The function $W_i$ with $i=0,\dots,3$
\begin{enumerate}
\item vanishes in all vertex Hermite \eqref{eq:DOFShermite} and edge-gradient degrees of freedom \eqref{eq:EdgeDofs},
\item satisfies $\partial_{\nu_f} W_i|_f = 0$ for all faces $f$ of $T$,
\item satisfies $W_i|_f = 0$ for all faces $f$ of $T$ with $f \neq f_i$.\label{itm:Casda}
\end{enumerate}
\end{lemma}
\begin{proof}
Let $i=0,\dots,3$.
As the bubble is zero along all edges, the vertex Hermite data  \eqref{eq:DOFShermite} of $B_i$ vanish. Moreover, if $e_{\ell m}\subset f_i$ and $n$ is determined by $\lbrace i,\ell,m,n\rbrace=\lbrace0,1,2,3\rbrace$, then Lemma~\ref{lem:edge-bubble-properties} yields
\begin{align*}
\nabla B_i|_{e_{\ell m}} = \lambda_\ell\lambda_m\nabla\lambda_n = \nabla E_{\ell m}^{n}|_{e_{\ell m}}.
\end{align*}
The remaining terms in the definition of $W_i$ have vanishing gradient on $e_{\ell m}$. On all other edges the gradient of $B_i$ and of the correction equal zero. Hence the degrees of freedom in \eqref{eq:DOFShermite} and \eqref{eq:EdgeDofs} vanish. The vanishing normal derivative on a face $f$ results from the fact that by design $\partial_{\nu_f} W_i|_f \in \mathcal{P}_2(f)$ is a quadratic polynomial which is, as shown in the first part of the proof, zero at all the vertices of $f$ and in the midpoints of all edges. This yields $\partial_{\nu_f} W_i|_f = 0$. The property in \ref{itm:Casda} follows by Lemma~\ref{lem:edge-bubble-properties} and the definition of $B_i$.
\end{proof}
We define the cubic singular Zienkiewicz as
\begin{align*}
  Z_3^{\operatorname{cub}}(T)\coloneqq Z_3^{\operatorname{sing}}(T)+\mathcal B_{\operatorname{cub}}(T) = Z_3^{\operatorname{sing}}(T)+\mathcal B_{\operatorname{cub}}^{\operatorname{cor}}(T).
\end{align*}
The additional degrees of freedom are the missing cubic facet functionals in Remark~\ref{rem:EquiChar}, which read for $z\in W^{2,\infty}(T)$ and $i=0,\dots,3$
\begin{align}\label{eq:CubicFaceDof}
  \mathcal C_i(z) \coloneqq 6 z(\operatorname{mid}(f_i)) - 2\sum_{j\in J_i}z(v_j)  +\sum_{j\in J_i}\nabla z(v_j)\cdot \big(v_j-\operatorname{mid}(f_i)\big).
\end{align}
This functional only depends on the scalar trace on $f_i$, since the vectors $v_j-\operatorname{mid}(f_i)$ are tangential to $\nu_i$. 
\begin{lemma}[Cubic Zienkiewicz tetrahedron]\label{lem:unisolvence-cubic}
The degrees of freedom consisting of the vertex Hermite data \eqref{eq:DOFShermite}, the edge-gradient data \eqref{eq:EdgeDofs}, and the facet functionals \eqref{eq:CubicFaceDof} are unisolvent for the space $Z_3^{\operatorname{cub}}(T)$. Moreover, we have
\begin{align*}
Z_3^{\operatorname{sing}}(T) \cap \mathcal{B}_{\operatorname{cub}}(T) = Z_3^{\operatorname{sing}}(T) \cap \mathcal{B}^{\operatorname{cor}}_{\operatorname{cub}}(T) = \lbrace0\rbrace
\quad\text{and}\quad \dim Z_3^{\operatorname{cub}}(T) = 32.
\end{align*}
\end{lemma}
\begin{proof}
Let $z_b = z+ b \in Z_3^{\operatorname{cub}}(T)$ with $z \in Z_3^{\operatorname{sing}}(T)$ and $b \in \mathcal B_{\operatorname{cub}}^{\operatorname{cor}}(T)$ vanish at all degrees of freedom.
Using Lemma~\ref{lem:CorrBubble} and the unisolvence of $Z_3^{\operatorname{sing}}(T)$, we conclude $z = 0$. Furthermore, we have for all $i,j=0,\dots,3$
\begin{align*}
\mathcal C_j(W_i) = \begin{cases}
0 &\text{if }j\neq i,\\
1/18 & \text{if }j = i.
\end{cases}
\end{align*}
Hence, if $z_b = b$ vanishes at the degrees of freedom in \eqref{eq:CubicFaceDof}, it must equal zero.   This yields unisolvence and linear independence of basis functions in $Z_3^{\operatorname{sing}}(T)$ and $(W_i)_{i=0}^3$.
\end{proof}
Let $\tria$ be a triangulation of $\Omega$. For every interior facet $f$ fix a unit normal vector $\nu_f$ and for every edge fix the vectors $\gamma_{\ell m}^0,\gamma_{\ell m}^1$. Define the global space
\begin{align*}
Z_{3}^{\operatorname{cub}}(\tria) \coloneqq \lbrace z\in L^2(\Omega)&\colon z|_T\in Z_3^{\operatorname{cub}}(T)\text{ for all }T\in\tria,\text{ and all global }\\
&\quad \text{degrees of freedom \eqref{eq:DOFShermite}, \eqref{eq:EdgeDofs}, \eqref{eq:CubicFaceDof} are single-valued}\rbrace.
\end{align*}
\begin{theorem}[$H^2$ conformity -- cubic]
\label{thm:unisolvence-cub}\ 
The cubic Zienkiewicz space satisfies
\begin{align*}
Z_{3}^{\operatorname{cub}}(\tria)\subset H^2(\Omega).
\end{align*}
\end{theorem}
\begin{proof}
Any $z_b \in Z_3^{\operatorname{cub}}(\tria)$ decomposes into functions $z_b = z + b$ such that
\begin{align*}
z|_T \in Z_3^{\operatorname{sing}}(T)\qquad \text{and}\qquad b|_T \in \mathcal B_{\operatorname{cub}}^{\operatorname{cor}}(T)\qquad\text{for all }T\in \tria.
\end{align*}
Using Lemma~\ref{lem:CorrBubble}, we conclude that $z$ is single valued at all global Hermite and edge-gradient degrees of freedom \eqref{eq:DOFShermite} and \eqref{eq:EdgeDofs}, proving that $z \in  Z_{3}^{\operatorname{sing}}(\tria) \subset H^2(\Omega)$.
According to Lemma~\ref{lem:CorrBubble}, the corrected cubic bubble part $b$ has a vanishing normal derivative on each face.
Continuity then follows from the fact that for any face $f$ shared by two simplices $T_+$ and $T_-$ only one local basis function $W_f^\pm$ in $\mathcal B_{\operatorname{cub}}^{\operatorname{cor}}(T_\pm)$ attains non-zero values which coincide due to \eqref{eq:CubicFaceDof} and $W_f^+|_f=W_f^-|_f$.
\end{proof}

\subsection{The reduced Zienkiewicz tetrahedron}\label{subsec:reduced-singular-zienkiewicz-tetrahedron}
In this subsection, we modify the singular Zienkiewicz space defined in Section~\ref{subbsec:SingZienk} to reduce its dimension.
Motivated by the corresponding idea in 2D, we therefore remove quadratic components in its normal traces. More precisely, we set 
\begin{align*}
Z_3^{\operatorname{red}}(T) \coloneqq \lbrace z \in Z_3^{\operatorname{sing}}(T) \colon \partial_{\nu_i} z|_{f_i}\in \mathcal P_1(f_i)\text{ for all }i=0,\dots,3\rbrace.
\end{align*}
In other words, we use the nodal basis functions associated to the edge-gradient degrees of freedom \eqref{eq:EdgeDofs} to correct the nodal basis functions. At the same time, we exclude them from the basis, resulting in 
\begin{align*}
\dim Z_3^{\operatorname{red}}(T) = \dim Z_3^{\operatorname{sing}}(T) - 12 = 16.
\end{align*}
The degrees of freedom of the reduced singular Zienkiewicz tetrahedron are only the vertex Hermite data in \eqref{eq:DOFShermite}.
\begin{lemma}[Unisolvence of the reduced singular Zienkiewicz tetrahedron]\label{lem:unisolvence-reduced-singular-zienkiewicz}
The vertex Hermite degrees of freedom are unisolvent for $Z_3^{\operatorname{red}}(T)$.
\end{lemma}
\begin{proof}
Let $z \in Z_3^{\operatorname{red}}(T)$ be a function that is zero in all Hermite degrees of freedom \eqref{eq:DOFShermite}. By definition, $z \in Z_3^{\operatorname{sing}}(T)$ with affine normal derivatives at the edges, implying that the edge-gradient degrees of freedom \eqref{eq:EdgeDofs} equal zero, too. Hence, unisolvence follows by the unisolvence of $Z_3^{\operatorname{sing}}(T)$ stated in Lemma~\ref{lem:unisolvence-singular-reduced}.
\end{proof}
Let $\tria$ be a triangulation of $\Omega$ and define the global space
\begin{align*}
Z_{3}^{\operatorname{red}}(\tria) \coloneqq \lbrace z\in L^2(\Omega)&\colon z|_T\in Z_3^{\operatorname{red}}(T)\text{ for all }T\in\tria\text{ and all global }\\
&\quad \text{degrees of freedom \eqref{eq:DOFShermite} are single-valued}\rbrace.
\end{align*}
\begin{lemma}[$H^2$ conformity -- reduced]
The reduced Zienkiewicz space satisfies
\begin{align*}
Z_{3}^{\operatorname{red}}(\tria)\subset H^2(\Omega).
\end{align*}
\end{lemma}
\begin{proof}
Continuity and continuous differentiability follows by arguments similar to those in the proof of Lemma~\ref{lem:unisolvence-reduced-singular-zienkiewicz}; that is, one observes that any $z\in Z_{3}^{\operatorname{red}}(\tria)$ has also single-valued edge-gradient degrees of freedom \eqref{eq:EdgeDofs} and is thus an element in $Z_{3}^{\operatorname{sing}}(\tria) \subset H^2(\Omega)$.
\end{proof}

\begin{remark}[Polynomial inclusion]
The reduced singular Zienkiewicz space still contains all quadratic polynomials; that is, $\mathcal P_2(T)\subset Z_3^{\operatorname{red}}(T)$.
Indeed, if $p\in\mathcal P_2(T)$, then $\nabla p$ is affine on $T$.
\end{remark}

\section{Integration formula for rational functions}\label{sec:3d-rational-moments}

This section develops an exact integration formula for the rational monomials $R_\beta^\alpha$ defined in \eqref{eq:DefRatBubble}. This allows for an implementation of the singular Zienkiewicz element for the biharmonic problem in 3D following the implementation discussed for the 2D setting in \cite{DieningStornTscherpel24}. Throughout this section let $T=[v_0,v_1,v_2,v_3]\subset\mathbb R^3$ be a tetrahedron with vertices $v_0,\dots,v_3$ and associated barycentric coordinates $\lambda_0,\dots,\lambda_3$. 
We aim at computing the normalized moment
\begin{align*}
\mathcal M_\beta^\alpha(T)\coloneqq \dashint_T R_\beta^\alpha(\lambda(x))\,\mathrm{d}x \coloneqq |T|^{-1}\int_T R_\beta^\alpha(\lambda(x))\,\mathrm{d}x.
\end{align*}
Note that an affine transformation yields the independence of $\mathcal M_\beta^\alpha(T)$ on $T$. We hence skip the argument $T$; that is, we write $\mathcal M_\beta^\alpha \coloneqq \mathcal M_\beta^\alpha(T)$.
In 2D the moment could be computed iteratively by exploiting relations of the type $(\lambda_0 + \lambda_1) = 1 - \lambda_2$ which allowed for the representation of the moment as the sum of computable moments \cite{DieningStornTscherpel24}. This section follows a similar strategy. The flowchart in Figure~\ref{fig:rational-moment-flowchart} at the end of this section illustrates the resulting routine.
\subsection{Reduction to triangular and star moments}
\label{subsec:direct-triangle-star-reduction}
Let $I = \lbrace 0,1,2,3\rbrace$. 
For pairwise distinct indices $a,b,c,d\in I$, let $e^\ast=[v_c,v_d]$ denote the edge opposite to $e=[v_a,v_b]$. Moreover, let $\fre_e \coloneqq \fre_{ab}\in\mathbb N_0^{\mathcal E_4}$ denote the unit multi-index corresponding to the edge $e$.
\begin{lemma}[Opposite-edge reduction]
\label{lem:opposite-edge-reduction}
Let $e,e^\ast\in\mathcal E_4$ be opposite edges. If $1 \leq \beta_e$ and $1 \leq \beta_{e^\ast}$, then
\begin{align*}
  \mathcal M_\beta^\alpha = \mathcal M_{\beta-\fre_e}^\alpha + \mathcal M_{\beta-\fre_{e^\ast}}^\alpha.
\end{align*}
If the moment $  \mathcal M_\beta^\alpha$ is finite, the moments $\mathcal M_{\beta-\fre_e}^\alpha$  and $\mathcal M_{\beta-\fre_{e^\ast}}^\alpha$ are finite as well. 
\end{lemma}
\begin{proof}
Let $e = [a,b]$ and  $e^\ast= [c,d]$ be opposite edges and set $s_e\coloneqq \lambda_a+\lambda_b$ and $s_{e^\ast} \coloneqq \lambda_c + \lambda_d$.
Since $s_e+s_{e^\ast}=1$, we obtain
\begin{align*}
  R_\beta^\alpha = (s_e+s_{e^\ast})R_\beta^\alpha = R_{\beta-\fre_e}^\alpha + R_{\beta-\fre_{e^\ast}}^\alpha .
\end{align*}
Integration over $T$ proves the identity. Boundedness of the moments on the right-hand side follows by the fact that reducing the power $\beta_e$ and $\beta_{e^\ast}$ reduces the singularities of rational functions.
\end{proof}
The formula in the lemma allows us to represent the moment as the sum of moments with smaller exponents in the denominators. This process can be repeated iteratively until we reach the following configuration. Its definition involves the notion of the support of a multi-index $\beta \in \mathbb{N}_0^{\mathcal{E}_4}$, which is the set of indices 
\begin{align*}
\operatorname{supp}(\beta) \coloneqq \lbrace ab \in \mathcal{E}_4 \colon \beta_{ab} > 0 \rbrace.
\end{align*} 
\begin{definition}[Reduced multi-index] 
We call $\beta \in \mathbb{N}_0^{\mathcal{E}_4}$ reduced if no two opposite edges occur simultaneously in its support; that is, for $\lbrace a,b,c,d\rbrace = \lbrace 0,1,2,3\rbrace$ 
\begin{align*}
ab \in\operatorname{supp}(\beta)\qquad \text{implies}\qquad  cd \not\in \operatorname{supp}(\beta).
\end{align*}
\end{definition}
Applying Lemma~\ref{lem:opposite-edge-reduction} iteratively allows us to write each finite moment $\mathcal M_\beta^\alpha$ is a finite linear combination of finite moments $\mathcal M_{\beta'}^\alpha$ with reduced multi-indices $\beta' \in \mathbb{N}_0^{\mathcal{E}_4}$.
It remains to consider the special case of reduced multi-indices $\beta \in \mathbb{N}_0^{\mathcal{E}_4}$.

\begin{lemma}[Reduced denominator graphs]\label{lem:reduced-index} 
Let $\beta\in\mathbb N_0^{\mathcal E_4}$ be reduced. Then there exist
pairwise distinct indices $\lbrace a,b,c,d \rbrace = \lbrace 0, 1, 2, 3\rbrace$ such that
\begin{align*}
  \operatorname{supp}(\beta) \subset\lbrace ab,bc,ca\rbrace
\qquad \text{or}\qquad  \operatorname{supp}(\beta)
  \subset\lbrace ab,ac,ad\rbrace.
\end{align*}
We call the first case ``one face'' and the second one ``vertex star''.
\end{lemma}

\begin{proof}
If the support of $\beta$ contains more than three elements, it must contain two opposite edges. Hence, $\operatorname{supp}(\beta)$ has less than four elements. All possible reduced cases are then covered by the one face and vertex star.
\end{proof}

\subsubsection{One face}
To investigate the one face case, we set the 2D reference simplex
\begin{align*}
  T_2\coloneqq\lbrace \eta=(\eta_1,\eta_2)\in\mathbb R^2\colon 0\leq  \eta_1,\ 0 \leq \eta_2,\ \eta_1+\eta_2 \leq 1 \rbrace .
\end{align*}
For an ordered triple of distinct indices $a,b,c\in\lbrace 0,1,2,3\rbrace$, we define the barycentric coordinates
\begin{align*}
  \mu_a(\eta)\coloneqq1-\eta_1-\eta_2,\qquad  \mu_b(\eta)\coloneqq\eta_1,\qquad \mu_c(\eta)\coloneqq\eta_2.
\end{align*}
For $\gamma=(\gamma_a,\gamma_b,\gamma_c)\in\mathbb N_0^3$ and $\beta_F=(\beta_{ab},\beta_{ac},\beta_{bc})\in\mathbb N_0^3$ with face $F = [v_a,v_b,v_c]$, define the normalized two-dimensional moment
\begin{align*}
  \mathcal I_{\beta_F}^{\gamma} &\coloneqq  \dashint_{T_2}\frac{\mu_a(\eta)^{\gamma_a}\mu_b(\eta)^{\gamma_b}\mu_c(\eta)^{\gamma_c}}{(\mu_a(\eta)+\mu_b(\eta))^{\beta_{ab}}(\mu_a(\eta)+\mu_c(\eta))^{\beta_{ac}}(\mu_b(\eta)+\mu_c(\eta))^{\beta_{bc}}}  \,\mathrm d\eta \\
  &\, = \dashint_{T_2}\frac{\mu_a(\eta)^{\gamma_a}\mu_b(\eta)^{\gamma_b}\mu_c(\eta)^{\gamma_c}}{(1-\mu_c(\eta))^{\beta_{ab}}(1-\mu_b(\eta))^{\beta_{ac}}(1-\mu_a(\eta))^{\beta_{bc}}}  \,\mathrm d\eta.
\end{align*}
This moment can be evaluated exactly with the algorithm in \cite[Algo.~2]{DieningStornTscherpel24}.
\begin{lemma}[One face]
\label{lem:triangular-moments}
Let $\alpha \in \mathbb{N}_0^4$ and let $\beta \in \mathbb{N}_0^{\mathcal{E}_4}$ with $\operatorname{supp}\beta\subset\lbrace ab,ac,bc\rbrace$ for three distinct indices $a,b,c\in\lbrace0,1,2,3\rbrace$. Let $d$ be a remaining index and set
\begin{align*}
  \gamma\coloneqq(\alpha_a,\alpha_b,\alpha_c), \qquad  |\gamma|\coloneqq\alpha_a+\alpha_b+\alpha_c, \qquad  |\beta_F|\coloneqq\beta_{ab}+\beta_{ac}+\beta_{bc}.
\end{align*}
If $\mathcal M_\beta^\alpha$ is finite, then $|\beta_F|\leq |\gamma|+2$ and we have
\begin{align*}
  \mathcal M_\beta^\alpha = 3\, \frac{(|\gamma|-|\beta_F|+2)!\,\alpha_d!}{(|\gamma|-|\beta_F|+\alpha_d+3)!}\,  \mathcal I_{\beta_F}^{\gamma}.
\end{align*}
\end{lemma}
\begin{proof}
Let the moment $\mathcal M_\beta^\alpha$ be finite with distinct vertices  $a,b,c\in\lbrace0,1,2,3\rbrace$ satisfying $\operatorname{supp}\beta\subset\lbrace ab,ac,bc\rbrace$.

\textit{Step 1 (Bound for $|\beta_F|$).} 
If $|\beta_F| =0$, the bound $|\beta_F| \leq |\gamma|+2$ follows trivially. If $0< |\beta_F|$, we apply Lemma~\ref{lem:Regularity} with $m=0$ and $p=1$ and set $S\coloneqq\lbrace a,b,c\rbrace$.
Then $\operatorname{supp}\beta\subset\lbrace ab,ac,bc\rbrace$ yields
\begin{align*}
  A_S(\alpha)=\alpha_a+\alpha_b+\alpha_c=|\gamma|\qquad \text{and} \qquad B_S(\beta)=\beta_{ab}+\beta_{ac}+\beta_{bc}=|\beta_F|.
\end{align*}
The regularity criterion gives
\begin{align*}
  -|S| < A_S(\alpha)-B_S(\beta)\qquad\text{which equals}\qquad  -3 < |\gamma|-|\beta_F|.
\end{align*}
Since the quantity $|\gamma|-|\beta_F|$ is an integer, we obtain
\begin{align*}
  |\beta_F|\leq |\gamma|+2.
\end{align*}

\textit{Step 2 (Identity).}
Set the affine parametrization  $y_F\colon T_2 \to F\coloneqq[v_a,v_b,v_c]$ of the face $F$ over $T_2$ with
\begin{align*}
y_F(\eta)\coloneqq\mu_a(\eta)v_a+\mu_b(\eta)v_b+\mu_c(\eta)v_c\qquad\text{for all }\eta \in T_2.
\end{align*}
We parametrize the simplex $T$ by $\Phi\colon (0,1)\times T_2 \to T$ with 
\begin{align*}
 \Phi(\rho,\eta)\coloneqq (1-\rho)v_d + \rho y_F(\eta)\qquad\text{for all }(\rho,\eta)\in (0,1)\times T_2.
\end{align*}
The barycentric coordinates of any $x=\Phi(\rho,\eta) \in T$ read
\begin{align*}
  \lambda_a(x)=\rho\mu_a(\eta),\qquad
  \lambda_b(x)=\rho\mu_b(\eta),\qquad
  \lambda_c(x)=\rho\mu_c(\eta),\qquad
  \lambda_d(x)=1-\rho.
\end{align*}
Moreover, we have for all $i,j\in\lbrace a,b,c\rbrace$ the identity
\begin{align*}
  \lambda_i(x)+\lambda_j(x)=\rho(\mu_i(\eta)+\mu_j(\eta)).
\end{align*}
A direct calculation (or arguing that $|\lbrace \Phi(\rho,\eta) \colon \eta \in T_2\rbrace| = |F| \rho^2$ for the two-dimensional slice) shows that the determinant of the Jacobian of $\Phi$ equals
\begin{align*}
  |\det D\Phi(\rho,\eta)|=6|T|\rho^2.
\end{align*}
The change of variables formula and the assumption $\operatorname{supp}\beta\subset\lbrace ab,ac,bc\rbrace$ thus yield
\begin{align*}
  \mathcal M_\beta^\alpha   &= 3\int_0^1 \rho^{|\gamma|-|\beta_F|+2}(1-\rho)^{\alpha_d} \,\mathrm d\rho\, \mathcal I_{\beta_F}^{\gamma}.
  \end{align*}
Since we have $0 \leq |\gamma|-|\beta_F|+2$ for integrable moments $\mathcal{M}_\beta^\alpha$, the exponent is non-negative and we can apply the classical formula for the integration of barycentric coordinates, resulting in
\begin{align*}
\int_0^1 \rho^{|\gamma|-|\beta_F|+2}(1-\rho)^{\alpha_d} \,\mathrm d\rho =   \frac{(|\gamma|-|\beta_F|+2)!\,\alpha_d!}{(|\gamma|-|\beta_F|+\alpha_d+3)!}.
\end{align*}
Combining the identities concludes the proof.
\end{proof}

\subsubsection{Vertex-star}
It remains to discuss the vertex-star situation. By relabeling it suffices to consider a star centered at the vertex $v_0$.

\begin{definition}[Star moments]
\label{def:star-moments}
For $\alpha\in\mathbb N_0^4$ and $\gamma=(\gamma_1,\gamma_2,\gamma_3)\in\mathbb N_0^3$, define
\begin{align*}
  \mathcal S^\alpha_\gamma \coloneqq \dashint_T \frac{\lambda_0^{\alpha_0}\lambda_1^{\alpha_1}\lambda_2^{\alpha_2}\lambda_3^{\alpha_3}} {(\lambda_0+\lambda_1)^{\gamma_1}(\lambda_0+\lambda_2)^{\gamma_2}(\lambda_0+\lambda_3)^{\gamma_3}} \dx.
\end{align*}
\end{definition}

The regularity criterion in Lemma~\ref{lem:Regularity} specializes as follows.
\begin{lemma}[Integrability]\label{lem:integrability}
The moment $\mathcal S^\alpha_\gamma$ is finite if and only if
\begin{align}
  \gamma_i&\leq \alpha_0+\alpha_i+1 &&\text{for all }i=1,2,3,
  \label{eq:star-single-criterion}\\
  \gamma_i+\gamma_j&\leq \alpha_0+\alpha_i+\alpha_j+2 &&\text{for all }1\leq i<j\leq3 .
  \label{eq:star-pair-criterion}
\end{align}
\end{lemma}
\begin{proof}
We write the star denominator in the notation of Lemma~\ref{lem:Regularity} by setting
\begin{align*}
  \beta_{01}\coloneqq\gamma_1, \qquad \beta_{02}\coloneqq\gamma_2, \qquad \beta_{03}\coloneqq\gamma_3.
\end{align*}
All remaining denominator exponents equal zero. We apply Lemma~\ref{lem:Regularity} with $m=0$ and $p=1$.
The criterion states that, for every $S\subset\lbrace0,1,2,3\rbrace$ with $2\leq |S|\leq3$ and $B_S(\beta)>0$, there holds
\begin{align*}
  -|S|<A_S(\alpha)-B_S(\beta).
\end{align*}
Since $A_S(\alpha)-B_S(\beta)$ is an integer, this is equivalent to
\begin{align*}
  B_S(\beta)\leq A_S(\alpha)+|S|-1.
\end{align*}
Considering the case $S=\lbrace 0,i \rbrace$ with $i=1,2,3$ then yields \eqref{eq:star-single-criterion}. Considering $S=\lbrace 0,i,j \rbrace$ with $1\leq i<j \leq 3$ yields \eqref{eq:star-pair-criterion}.
\end{proof}
We next introduce admissible reductions for star moments, where the word admissible refers to the fact  that every moment generated by the reduction of a finite moment is again finite. We refer to an inequality in \eqref{eq:star-single-criterion} or \eqref{eq:star-pair-criterion} as critical, if the inequality holds with equality.
\begin{lemma}[Star reduction]\label{lem:simple-admissible-star-reduction}
Let $\mathcal S^\alpha_\gamma$ be finite and assume that $1 \leq \alpha_0$. 
Suppose that there exists an index $k\in\lbrace1,2,3\rbrace$ with $1 \leq \gamma_k$ such that every critical inequality in \eqref{eq:star-single-criterion}--\eqref{eq:star-pair-criterion} involves the index $k$.
Then we have 
\begin{align*}
  \mathcal S^\alpha_\gamma = \mathcal S^{\alpha-\fre_0}_{\gamma-\fre_k} - \mathcal S^{\alpha-\fre_0+\fre_k}_{\gamma}.
\end{align*}
Both moments on the right-hand side are finite.
\end{lemma}

\begin{proof}
Multiplying $\lambda_0=(\lambda_0+\lambda_k)-\lambda_k$ by $\lambda^{\alpha-\fre_0}$ and dividing by the star denominator gives the displayed formula after integration over $T$.
Integrability of the resulting moments then follows by the assumptions and Lemma~\ref{lem:integrability}.
\end{proof}

We cannot apply the reduction step in Lemma~\ref{lem:simple-admissible-star-reduction}, when the critical conditions cannot be hit by a single index.
This motivates the following decomposition. 

\begin{lemma}[Critical star reduction]
\label{lem:critical-star-reduction}
Let $\mathcal S^\alpha_\gamma$ be finite and assume that $1 \leq \alpha_0$. Suppose that all pair conditions \eqref{eq:star-pair-criterion} are critical; that is,
\begin{align*}
  \gamma_i+\gamma_j=\alpha_0+\alpha_i+\alpha_j+2
  \qquad\text{for all }1\leq i<j\leq3 .
\end{align*}
Then we have the identity
\begin{align*}
  \mathcal S^\alpha_\gamma &=
  \frac13\sum_{1\leq i<j\leq 3}\mathcal S^{\alpha-\fre_0}_{\gamma-\fre_i-\fre_j}
  +\frac16\sum_{\substack{1 \leq i,j\leq 3\\ i\neq j}} \mathcal S^{\alpha-\fre_0+\fre_i}_{\gamma-\fre_j} -\frac23\sum_{1 \leq i<j \leq 3}  \mathcal S^{\alpha-\fre_0+\fre_i+\fre_j}_{\gamma}.
\end{align*}
All moments on the right-hand side are finite.
\end{lemma}

\begin{proof}
Set $s\coloneqq\lambda_1+\lambda_2+\lambda_3$ and $q\coloneqq\sum_{1\leq i<j\leq3}\lambda_i\lambda_j$.
The barycentric partition of unity gives $\lambda_0+s=1$ and consequently
\begin{align*}
  \lambda_0=\lambda_0(\lambda_0+s)=\lambda_0^2+\lambda_0s.
\end{align*}
We rewrite
\begin{align*}
  \lambda_0^2+\lambda_0s =\frac13\left(3\lambda_0^2+2\lambda_0s+q\right) +   \frac16\left(2\lambda_0s+2q\right) -\frac23 q.
\end{align*}
The first expression in parentheses satisfies
\begin{align*}
  3\lambda_0^2+2\lambda_0s+q = \sum_{1\leq i<j\leq3}(\lambda_0+\lambda_i)(\lambda_0+\lambda_j).
\end{align*}
The second expression satisfies
\begin{align*}
2\lambda_0s+2q =\sum_{\substack{1\leq i,j\leq3\\i\neq j}}\lambda_i(\lambda_0+\lambda_j).
\end{align*}
This yields the identity
\begin{align*}
  \lambda_0 = \frac13\sum_{1\leq i<j\leq 3}(\lambda_0+\lambda_i)(\lambda_0+\lambda_j) +\frac16\sum_{\substack{1\leq i,j\leq3 \\ i\neq j}}\lambda_i(\lambda_0+\lambda_j) -\frac23 \sum_{1\leq i<j \leq 3}\lambda_i\lambda_j.
\end{align*}
Multiplying by $\lambda^{\alpha-\fre_0}$ and dividing by the star denominator gives the stated formula.
In every term on the right-hand side the central numerator power is reduced by one. Since all pair conditions are critical, this loss must be compensated for every pair $\lbrace i,j\rbrace$. In the first sum the corresponding denominator powers are reduced. In the second sum either a denominator power is reduced or a leaf numerator power is increased. In the third sum the two leaf numerator powers are increased. Thus all pair conditions remain valid.
To verify the single conditions \eqref{eq:star-single-criterion}, we subtract the critical pair identities, leading to
\begin{align*}
  \gamma_i=\alpha_i+1+\frac{\alpha_0}{2} \in \mathbb{Z}\qquad\text{for all }i=1,2,3.
\end{align*}
This shows that $\alpha_0 \geq 1$ is even, implying $1+\alpha_0/2 \leq \alpha_0$ and consequently $\gamma_i\leq \alpha_0+\alpha_i$.
Hence, the new rational functions satisfy the single condition \eqref{eq:star-single-criterion}.
\end{proof}
The following lemma shows that we can apply Lemma~\ref{lem:simple-admissible-star-reduction} or Lemma~\ref{lem:critical-star-reduction} iteratively. 
The iteration continues until either $\alpha_0=0$ or $\gamma=0$. In the latter case, the denominator has disappeared and the moment is a polynomial moment.

\begin{lemma}[Reduction to boundary-star moments]
\label{lem:reduction-to-boundary-star}
If $\mathcal S^\alpha_\gamma$ is a finite moment with $1\leq \alpha_0$ and $\gamma_k \geq 1$ for some $k=1,2,3$, the assumptions of Lemma~\ref{lem:simple-admissible-star-reduction} or~\ref{lem:critical-star-reduction} are satisfied.
\end{lemma}
\begin{proof}
Let $\mathcal S^\alpha_\gamma$ be a finite moment with $1\leq \alpha_0$.
Suppose the single condition \eqref{eq:star-single-criterion} is attained for $k=1$; that is,
\begin{align*}
\gamma_1 = \alpha_0 + \alpha_1 +1.
\end{align*} 
Then this identity and the pair condition \eqref{eq:star-pair-criterion} yield for $i  =2,3$ that
\begin{align*}
 \gamma_i \leq  \alpha_i + 1 < \alpha_0 + \alpha_i + 1.
\end{align*}
Hence, the single condition  is not critical for $i=2,3$. Moreover, the bound yields
\begin{align*}
\gamma_2 + \gamma_3\leq \alpha_2 + \alpha_3 + 2 < \alpha_0 + \alpha_2 + \alpha_3 + 2. 
\end{align*}
Hence, the pair condition is not critical for the pair $\lbrace 2,3\rbrace$, which shows that the assumptions of Lemma~\ref{lem:simple-admissible-star-reduction} are satisfied whenever a single condition \eqref{eq:star-single-criterion} is critical.
It remains to cover the cases where all single conditions are not critical. If we have strictly less than three critical pair conditions \eqref{eq:star-pair-criterion}, at least one index occurs in all of them, allowing for the application of Lemma~\ref{lem:simple-admissible-star-reduction}. The remaining case has three critical pair conditions, which is exactly the assumption in Lemma~\ref{lem:critical-star-reduction}.
\end{proof}
It remains to investigate case $\alpha_0 = 0$. 
\subsection{Star moment with $\alpha_0 = 0$}\label{subsec:boundary-star-routine}
This subsection develops an integration formula for moments with multi-indices $\bar\alpha=(\alpha_1,\alpha_2,\alpha_3)\in\mathbb N_0^3$ and $\gamma=(\gamma_1,\gamma_2,\gamma_3)\in\mathbb N_0^3$
\begin{align*}
  \mathcal B^{\bar\alpha}_{\gamma} \coloneqq S_\gamma^{(0,\alpha_1,\alpha_2,\alpha_3)}  \coloneqq  \dashint_T \frac{\lambda_1^{\alpha_1}\lambda_2^{\alpha_2}\lambda_3^{\alpha_3}}{(\lambda_0+\lambda_1)^{\gamma_1}(\lambda_0+\lambda_2)^{\gamma_2}(\lambda_0+\lambda_3)^{\gamma_3}} \dx.
\end{align*}
The integrability criterion in \eqref{eq:star-single-criterion} reads
\begin{align*}
  \gamma_j\leq \alpha_j+1 \qquad\text{for all }j=1,2,3.
\end{align*}
This yields the remaining pair conditions \eqref{eq:star-pair-criterion}. We define
\begin{align*}
  \delta_j\coloneqq\alpha_j+1-\gamma_j\in\mathbb N_0 \qquad\text{and}\qquad |\delta|\coloneqq\delta_1+\delta_2+\delta_3.
\end{align*}
We consider the case $1 \leq \gamma$, since if one of the $\gamma_j$ vanishes, the moment is triangular and already covered by Lemma~\ref{lem:triangular-moments}.

\begin{lemma}[Critical star descent]\label{lem:boundary-star-descent}
Let $\mathcal B^{\bar\alpha}_{\gamma}$ be finite and assume that $1 \leq \gamma$. Let $k\in\lbrace 1,2,3 \rbrace$ with $\delta_k>0$. Then we have the identity
\begin{align*}
  \mathcal B^{\bar\alpha}_{\gamma} = \frac{\alpha_k}{|\delta|}\,\mathcal B^{\bar\alpha-\fre_k}_{\gamma} - \frac{\gamma_k}{|\delta|}\, \mathcal B^{\bar\alpha}_{\gamma+\fre_k}.
\end{align*}
Both moments on the right-hand side are finite.
\end{lemma}
\begin{proof}
The following proof combines, similar to the approach for polytopes in \cite[Rem.~2.5]{Lasserre98}, homogeneity and the divergence theorem.
Let $1\leq \gamma$ and $\delta_k>0$ for $k\in \lbrace 1,2,3 \rbrace$. We set the reference simplex $\widetilde T\coloneqq\lbrace z\in\mathbb R_{\geq0}^4\colon z_0+z_1+z_2+z_3=1\rbrace$ and define the function 
\begin{align*}
 f(z) \coloneqq \frac{z_1^{\alpha_1}z_2^{\alpha_2}z_3^{\alpha_3}}{(z_0+z_1)^{\gamma_1}(z_0+z_2)^{\gamma_2}(z_0+z_3)^{\gamma_3}}\qquad\text{for all }z \in \widetilde{T}.
\end{align*}
An integration by substitution reveals $\mathcal B^{\bar\alpha}_{\gamma} = \dashint_{\widetilde T} f(z)  \,\mathrm dz$.
We set the vector field 
\begin{align*}
  X_k(z)\coloneqq e_k-z \qquad\text{for all }z\in \widetilde{T}.
\end{align*}
Since $\delta_k>0$ and $\gamma_k\geq1$, we have $\alpha_k\geq1$. Hence, the function $f$ equals zero on the face $\lbrace z \in \widetilde{T}\colon z_k=0\rbrace$. 
The vector field $X_k$ is tangent to the affine hyperplane containing $\widetilde T$. Moreover, it is tangent to every face $\lbrace z_i=0\rbrace$ with $i\neq k$. Hence the boundary flux of $fX_k$ is zero. Gau\ss's divergence theorem, $\operatorname{div}_{\widetilde T}X_k=-3$, and $X_k$ being tangent to $\widetilde{T}$ yield
\begin{align*}
  0 = \int_{\widetilde T}\operatorname{div}_{\widetilde T}(f X_k)\,\mathrm dz =  \int_{\widetilde T} X_k\cdot\nabla f\,\mathrm dz - 3 \int_{\widetilde T} f\,\mathrm dz.
\end{align*}
The function $f$  satisfies $\rho^{|\delta|-3} f(z) = f(\rho z)$ for all $z\in \mathbb{R}^4$ and $\rho > 0$; that is, it is homogeneous of degree $q\coloneqq |\bar\alpha|-|\gamma|=|\delta|-3$.
Hence, the definition of $X_k$ and Euler's identity for homogeneous functions imply 
\begin{align*}
  X_k\cdot\nabla f =  \partial_{z_k}f-\sum_{i=0}^3 z_i\partial_{z_i}f = \partial_{z_k}f-qf .
\end{align*}
Combining these observations yields
\begin{align*}
  \int_{\widetilde T}\partial_{z_k}f\,\mathrm dz = (q+3)\int_{\widetilde T}f\,\mathrm dz = |\delta|\int_{\widetilde T}f\,\mathrm dz.
\end{align*}
The derivative reads
\begin{align*}
  \partial_{z_k}f(z) & = \alpha_k \frac{z_1^{\alpha_1} z_2^{\alpha_2}  z_3^{\alpha_3}}   {z_k(z_0+z_1)^{\gamma_1}(z_0+z_2)^{\gamma_2}(z_0+z_3)^{\gamma_3}} \\
  & \quad - \gamma_k \frac{z_1^{\alpha_1}z_2^{\alpha_2}z_3^{\alpha_3}}{(z_0+z_k)^{\gamma_k+1} \prod_{1\leq i\leq 3, i\neq k} (z_0+z_i)^{\gamma_i}}.
\end{align*}
Integrating this identity over $\widetilde T$ gives
\begin{align*}
  |\delta|\,\mathcal B^{\bar\alpha}_{\gamma} = \alpha_k\,\mathcal B^{\bar\alpha-\fre_k}_{\gamma} - \gamma_k\,\mathcal B^{\bar\alpha}_{\gamma+\fre_k}.
\end{align*}
Dividing by $|\delta|$ proves the formula.
Both moments on the right-hand side are finite, as by assumption $\delta_k > 0$, ensuring the single conditions \eqref{eq:star-single-criterion}. The pair conditions \eqref{eq:star-pair-criterion} then follow for the case $\alpha_0 = 0$ from the single conditions.
\end{proof}
An iterative application of Lemma~\ref{lem:boundary-star-descent} results in the representation of $\mathcal{B}_{\gamma'}^{\bar \alpha'}$ as linear combination of moments $\mathcal{B}_{\gamma}^{\bar \alpha}$ with $|\delta| = 0$; that is, moments that read with $a,b,c\in\mathbb N_0$
\begin{align*}
  \mathcal C_{a,b,c} \coloneqq \mathcal B^{(a,b,c)}_{(a+1,b+1,c+1)} = \dashint_T\frac{\lambda_1^a\lambda_2^b\lambda_3^c}{(\lambda_0+\lambda_1)^{a+1}(\lambda_0+\lambda_2)^{b+1}(\lambda_0+\lambda_3)^{c+1}} \dx.
\end{align*}
The following lemma, motivated by the Duffy transformation, transforms the integral into an integral over the cube. The transformation involves the function 
\begin{align}\label{eq:DefDxyz}
D(x,y,z)\coloneqq 1-xy-xz-yz+2xyz\qquad\text{for all }x,y,z\in [0,1].
\end{align}
\begin{lemma}[Cube representation]
\label{lem:diagonal-boundary-star-cube}
For all $a,b,c\in\mathbb N_0$ there holds
\begin{align*}
  \mathcal C_{a,b,c} = 6\int_0^1\int_0^1\int_0^1 \frac{x^a y^b z^c}{D(x,y,z)} \,\mathrm dz\,\mathrm dy\dx.
\end{align*}
\end{lemma}
\begin{proof}
We set the variables
\begin{align*}
  x\coloneqq \frac{\lambda_1}{\lambda_0+\lambda_1}, \qquad
  y\coloneqq \frac{\lambda_2}{\lambda_0+\lambda_2}, \qquad
  z\coloneqq \frac{\lambda_3}{\lambda_0+\lambda_3}.
\end{align*}
From the definition of $x$ we have $x(\lambda_0+\lambda_1)=\lambda_1$ and $x\lambda_0=(1-x)\lambda_1$.
This identity and similar arguments yield
\begin{align*}
  \lambda_1=\frac{x}{1-x}\lambda_0,\qquad
  \lambda_2=\frac{y}{1-y}\lambda_0,\qquad
  \lambda_3=\frac{z}{1-z}\lambda_0.
\end{align*}
Using these identities and the partition of unity yields
\begin{align*}
  1 &= \lambda_0+\lambda_1+\lambda_2+\lambda_3 = \lambda_0\left(1+\frac{x}{1-x}+\frac{y}{1-y}+\frac{z}{1-z}\right)\\
   & = \lambda_0\left( \frac{1-xy-xz-yz+2xyz}{(1-x)(1-y)(1-z)} \right).
\end{align*}
With $D(x,y,z) \coloneqq 1-xy-xz-yz+2xyz$ we thus obtain
\begin{align*}
  \lambda_0 = \frac{(1-x)(1-y)(1-z)}{D(x,y,z)}.
\end{align*}
Similarly, we conclude that
\begin{align*}
\begin{aligned}
 \lambda_1&=\frac{x(1-y)(1-z)}{D(x,y,z)}, \qquad
  \lambda_2=\frac{y(1-x)(1-z)}{D(x,y,z)}, \qquad 
 \lambda_3=\frac{z(1-x)(1-y)}{D(x,y,z)}.
\end{aligned}
\end{align*}
A calculation reveals for the coordinate transformation $(x,y,z)\mapsto(\lambda_1,\lambda_2,\lambda_3)$ that
\begin{align*}
  \det\frac{\partial(\lambda_1,\lambda_2,\lambda_3)}{\partial(x,y,z)} = \frac{(1-x)^2(1-y)^2(1-z)^2}{D(x,y,z)^4}.
\end{align*}
Consequently, an integration by substitution verifies the lemma.
\end{proof}
In order to evaluate this integral over a cube, we exploit the following reduction formula, involving the quantity
\begin{align*}
  \mathcal P_{r,s,t} \coloneqq 6\int_0^1\int_0^1\int_0^1 x^r y^s z^t\,\mathrm dz\,\mathrm dy\dx = \frac{6}{(r+1)(s+1)(t+1)}\quad \text{for all }r,s,t\in \mathbb N_0.
\end{align*}
\begin{lemma}[Reduction of $\mathcal C_{a,b,c}$]\label{lem:cube-reduction-one-zero-exponent}
We have for all positive integers $a,b,c\in\mathbb N$ that
\begin{align*}
  2\mathcal C_{a,b,c} = \mathcal P_{a-1,b-1,c-1} - \mathcal C_{a-1,b-1,c-1} + \mathcal C_{a,b,c-1} + \mathcal C_{a,b-1,c} + \mathcal C_{a-1,b,c}.
\end{align*}
\end{lemma}
\begin{proof}
The identity $D(x,y,z)=1-xy-xz-yz+2xyz$ can be rewritten as
\begin{align*}
  2xyz=D(x,y,z)-1+xy+xz+yz .
\end{align*}
Multiplying this identity by $x^{a-1}y^{b-1}z^{c-1}/D(x,y,z)$ and integrating over the unit cube leads to the lemma.
\end{proof}
Repeated use of Lemma~\ref{lem:cube-reduction-one-zero-exponent} reduces every $\mathcal C_{a,b,c}$ to moments with at least one zero exponent. By symmetry it remains to consider moments of the form 
\begin{align*}
\mathcal C_{a,b,0} = 6\int_0^1\int_0^1 x^a y^b \int_0^1 \frac{1 }{D(x,y,z)} \,\mathrm dz\,\mathrm dy\dx.
\end{align*}
\begin{lemma}[Logarithmic representation]\label{lem:two-variable-log-representation}
For all $a,b\in\mathbb N_0$ there holds
\begin{align*}
  \mathcal C_{a,b,0} = 6\int_0^1\int_0^1 \frac{x^a y^b}{x+y-2xy} \log\left(\frac{1-xy}{(1-x)(1-y)}\right) \,\mathrm dy\dx.
\end{align*}
\end{lemma}
\begin{proof}
By \eqref{eq:DefDxyz} we have for all $x,y,z\in [0,1]$ the identity
\begin{align*}
  D(x,y,z) = 1-xy-z(x+y-2xy).
\end{align*}
It holds that
\begin{align*}
  \frac{\mathrm d}{\mathrm dz} \log\left(1-xy-z(x+y-2xy)\right) = -\frac{x+y-2xy}{1-xy-z(x+y-2xy)} .
\end{align*}
Hence, we obtain the identity
\begin{align*}
  (x+y-2xy) \int_0^1 \frac{1}{D(x,y,z)}\,\mathrm dz = \log\left(\frac{1-xy}{(1-x)(1-y)}\right).
\end{align*}
This identity, Lemma~\ref{lem:diagonal-boundary-star-cube} with $c=0$, and Fubini's theorem yield the assertion.
\end{proof}

In our last iterative step we reduce the moments to a one-parameter family of moments and explicit integrals
\begin{align}\label{eq:DefLpq}
  \mathcal L_{p,q} \coloneqq 6\int_0^1\int_0^1 x^p y^q\log\left(\frac{1-xy}{(1-x)(1-y)}\right) \,\mathrm dy\dx\qquad\text{for } p,q\in \mathbb{N}_0.
\end{align}
The exact value of this integral involves the quantities  
\begin{align}\label{eq:DefHr}
\zeta(2) \coloneqq \sum_{k=1}^\infty \frac{1}{k^{2}} = \frac{\pi^2}{6}\quad \text{and} \quad H_m \coloneqq \sum_{r=1}^m \frac1r, \ H^{(2)}_m \coloneqq \sum_{r=1}^m \frac{1}{r^2}\quad \text{for }m\in \mathbb{N}_0. 
\end{align}
\begin{lemma}[Evaluation of $\mathcal L_{p,q}$]
\label{lem:explicit-logarithmic-moments}
Let $p,q\in\mathbb N_0$ and set $A\coloneqq p+1$, $B\coloneqq q+1$. Then we have
\begin{align*}
  \mathcal L_{p,q} = 6\left(\frac{H_A+H_B}{AB} - S_{A,B}\right)\quad \text{with}\quad 
  S_{A,B}
  \coloneqq
  \begin{cases}
    \frac{H_A/A-H_B/B}{B-A} &\text{if } A\neq B,\\
    \frac{H_A}{A^2}-\frac{\zeta(2)-H_A^{(2)}}{A}& \text{if } A=B.
  \end{cases}
\end{align*}
\end{lemma}

\begin{proof}
Let $p,q\in\mathbb N_0$ and set $A\coloneqq p+1$, $B\coloneqq q+1$.
We split the logarithm in the integral \eqref{eq:DefLpq} into
\begin{align*}
  \log\left(\frac{1-xy}{(1-x)(1-y)}\right) = \log(1-xy)-\log(1-x)-\log(1-y).
\end{align*}
Hence, the integral splits into a sum of three addends. To evaluate the latter two addends, we exploit that for all $m\in \mathbb{N}_0$ the series expansion $\log(1-t) = -\sum_{j=1}^{\infty}t^j/j$  for $0\leq t<1$ yields
\begin{align*}
- \int_0^1 t^m\log(1-t)\dt& = \sum_{j=1}^{\infty}\frac1j\int_0^1 t^{m+j}\dt = \sum_{j=1}^{\infty}\frac{1}{j(m+j+1)} = \frac{H_{m+1}}{m+1}.
\end{align*}
Using this identity, we obtain 
\begin{align*}
-\int_0^1\int_0^1 x^p y^q\log(1-y) \,\mathrm dy\dx = \frac{H_B}{B} \int_0^1 x^p \dx = \frac{H_B}{AB}.
\end{align*}
Similarly, we obtain the corresponding value $H_A/(AB)$ for the second term.
Using the series $\log(1-xy)=-\sum_{j=1}^{\infty}x^j y^j/j$, we conclude for the remaining term 
\begin{align*}
  \int_0^1\int_0^1 x^p y^q\log(1-xy)\,\mathrm dy\dx = -\sum_{j=1}^{\infty}\frac{1}{j(j+A)(j+B)}.
\end{align*}
If $A\neq B$, we conclude that
\begin{align*}
  \sum_{j=1}^{\infty}\frac{1}{j(j+A)(j+B)} &= \frac{1}{B-A} \left(\sum_{j=1}^{\infty}\frac{1}{j(j+A)} - \sum_{j=1}^{\infty}\frac{1}{j(j+B)} \right)\\
  & = \frac{1}{B-A} \left(\frac{H_A}{A}-\frac{H_B}{B} \right).
\end{align*}
If $A=B$, then the partial fraction decomposition
\begin{align*}
  \frac{1}{r(r+A)^2} = \frac{1}{A^2}\left(\frac1r-\frac{1}{r+A}\right) - \frac{1}{A}\frac{1}{(r+A)^2}
\end{align*}
yields
\begin{align*}
  \sum_{r=1}^{\infty}\frac{1}{r(r+A)^2} = \frac{H_A}{A^2} - \frac{\zeta(2)-H_A^{(2)}}{A}.
\end{align*}
Combining the three contributions proves the formula.
\end{proof}
The value $\mathcal{L}_{p,q}$ occurs in the following reduction formula for $\mathcal C_{a,b,0}$ characterized in Lemma~\ref{lem:two-variable-log-representation}.
\begin{lemma}[Reduction to $\mathcal F_n$]\label{lem:two-index-reduction-to-Fn}
For all $a,b\in\mathbb N$ there holds
\begin{align*}
  \mathcal C_{a,b,0} = \frac12\mathcal C_{a,b-1,0} + \frac12\mathcal C_{a-1,b,0} - \frac12\mathcal L_{a-1,b-1}.
\end{align*}
\end{lemma}

\begin{proof}
The denominator $B(x,y)\coloneqq x+y-2xy$ in Lemma~\ref{lem:two-variable-log-representation} satisfies
\begin{align*}
  xy=\frac12 x+\frac12 y-\frac12 B(x,y).
\end{align*}
Multiplying this identity by $x^{a-1}y^{b-1}\log\big((1-xy)/\big((1-x)(1-y)\big)\big)/B(x,y)$ and integrating over $(0,1)^2$ gives
\begin{align*}
  \mathcal C_{a,b,0} &= \frac12 \mathcal C_{a,b-1,0} + \frac12\mathcal C_{a-1,b,0} - \frac12\mathcal L_{a-1,b-1} \qedhere.
\end{align*}
\end{proof}
Finally, the star moments reduce to moments that read with $n\in \mathbb{N}_0$
\begin{align}\label{eq:Fn}
  \mathcal F_n \coloneqq \mathcal{C}_{n,0,0} = 6\int_0^1\int_0^1 \frac{x^n}{x+y-2xy} \log\left(\frac{1-xy}{(1-x)(1-y)}\right) \,\mathrm dy\dx.
\end{align}
This integral can be evaluated for any $n\in \mathbb{N}_0$ by a formula derived in the remainder of this section.
\begin{lemma}[Representation as 1D integral]\label{lem:1DIntegralRepresneation}
Define for all $t \in (0,1)$ the function
\begin{align*}
  \Phi(t)\coloneqq \frac{\pi^2}{6}+2\operatorname{Li}_2(-t).
\end{align*}
Then we have the  representation
\begin{align*}
  \mathcal F_n = 6\int_0^1\left(\frac{1+t^n}{(1-t)(1+t)^{n+1}}\Phi(t)  + \frac{\log^2(t)}{(1-t)(1+t)^{n+1}} \right) \dt.
\end{align*}
\end{lemma}
\begin{proof}
Define the transformation
\begin{align*}
  x=\frac{a}{1+a}\qquad \text{and}\qquad y=\frac{b}{1+b}.
\end{align*}
It holds that 
\begin{align*}
  x+y-2xy = \frac{a+b}{(1+a)(1+b)}\quad \text{and}\quad \frac{1-xy}{(1-x)(1-y)} = 1+a+b.
\end{align*}
An integration by substitution of \eqref{eq:Fn} with $\mathrm dx\,\mathrm dy = (1+a)^{-2}(1+b)^{-2} \,\mathrm da\,\mathrm db$ yields
\begin{align}\label{eq:proofasdsadadsada}
  \mathcal F_n = 6\int_0^\infty \frac{a^n}{(1+a)^{n+1}}K(a)  \,\mathrm da\qquad\text{with } 
  K(a) \coloneqq \int_0^\infty \frac{\log(1+a+b)}{(a+b)(1+b)} \,\mathrm db.
\end{align}
We proceed with the computation of $K(a)$. We start with the case $a\in (0,1)$, where we set $t\coloneqq (a+b)/(1+b) \in (a,1)$. A calculation gives
\begin{align*}
  \frac{\mathrm db}{(a+b)(1+b)} = \frac{1}{1-a}\frac{\mathrm dt}{t}\qquad\text{and}  \qquad  1+a+b = \frac{1-at}{1-t}.
\end{align*}
Using this substitution, we obtain
\begin{align*}
  K(a) &= \frac{1}{1-a} \int_a^1 \frac{\log(1-at)-\log(1-t)}{t}\dt.
\end{align*}
The identities $\mathrm d\operatorname{Li}_2(t)/\mathrm dt = - \log(1-t)/t$ and $\operatorname{Li}_2(1) = \pi^2/6$ for the dilogarithm $\operatorname{Li}_2$ \cite[Chap.~2.6]{AndrewsAskeyRoy99} yield for the second term
\begin{align*}
\int_a^1 \frac{-\log(1-t)}{t}\dt = \left[\operatorname{Li}_2(t)\right]_{t=a}^{1} = \frac{\pi^2}{6}-\operatorname{Li}_2(a).
\end{align*}
Using the substitution $s=at$ and similar arguments as before, we conclude for the first term
\begin{align*}
  \int_a^1 \frac{\log(1-at)}{t}\dt = \int_{a^2}^{a}\frac{\log(1-s)}{s}\ds = \left[-\operatorname{Li}_2(s)\right]_{s=a^2}^{a} = \operatorname{Li}_2(a^2)-\operatorname{Li}_2(a).
\end{align*}
Combining these results shows 
\begin{align*}
  K(a) = \frac{\pi^2/6+\operatorname{Li}_2(a^2)-2\operatorname{Li}_2(a)}{1-a}\qquad\text{for all }a \in (0,1).
\end{align*}
It remains to consider the case $a \in (1,\infty)$. We first record a symmetry relation for $K$. Let $0<t<1$. From the definition of $K$ and the substitution $b=c/t$, we obtain
\begin{align*}
  K(1/t) &= \int_0^\infty \frac{\log(1+1/t+b)}{(1/t+b)(1+b)} \,\mathrm db = t\int_0^\infty \frac{\log(1+t+c)-\log(t)}{(t+c)(1+c)} \,\mathrm dc\\
  &= tK(t) - t\log(t) \int_0^\infty \frac{1}{(t+c)(1+c)} \,\mathrm dc.
\end{align*}
The remaining integral satisfies
\begin{align*}
  \int_0^\infty \frac{1}{(t+c)(1+c)} \,\mathrm dc = \frac{1}{1-t} \int_0^\infty \left(\frac1{t+c} - \frac1{1+c} \right) \,\mathrm dc  = -\frac{\log(t)}{1-t}.
\end{align*}
Therefore, we get
\begin{align*}
  K(1/t) = tK(t) + \frac{t\log^2(t)}{1-t} \qquad\text{for all }0<t<1.
\end{align*}
This identity and the substitution $a=1/t$ gives
\begin{align*}
  \int_1^\infty \frac{a^n}{(1+a)^{n+1}}K(a) \,\mathrm da & = \int_0^1 \frac{K(1/t)}{t(1+t)^{n+1}} \dt \\
  & = \int_0^1 \frac{K(t)}{(1+t)^{n+1}} \dt + \int_0^1 \frac{\log^2(t)}{(1-t)(1+t)^{n+1}} \dt.
\end{align*}
Combining the representations for $a < 1$ and $a>1$ with \eqref{eq:proofasdsadadsada} results in
\begin{align*}
  \mathcal F_n &= 6\int_0^1\left( \frac{(1+t^n) K(t)}{(1+t)^{n+1}} + \frac{\log^2(t)}{(1-t)(1+t)^{n+1}}\right) \dt.
\end{align*}
An application of the dilogarithm duplication identity $\operatorname{Li}_2(t^2) = 2\operatorname{Li}_2(t) + 2\operatorname{Li}_2(-t)$ \cite[Thm.~2.6.2]{AndrewsAskeyRoy99} concludes the proof.
\end{proof}
Using the 1D representation of $\mathcal{F}_n$, we conclude the following recurrence relation involving the constants $H_j$ defined in \eqref{eq:DefHr}.

\begin{lemma}[A one-step identity for the one-axis moments]\label{lem:one-axis-one-step-identity}
We have for all $n\in\mathbb N$ the identities
\begin{align*}
  \mathcal F_n-\frac12\mathcal F_{n-1} = 3\int_0^1 \frac{(1-t^{n-1})\Phi(t)+\log^2 t}{(1+t)^{n+1}} \dt = \frac{\pi^2}{2n} + \frac{6}{n}\sum_{j=1}^{n-1}\frac{H_j}{j}.
\end{align*}
\end{lemma}

\begin{proof}
Let $n\in \mathbb{N}$ and set
\begin{align*}
  A_n(t)\coloneqq \frac{1+t^n}{(1-t)(1+t)^{n+1}} \qquad\text{and}\qquad  B_n(t)\coloneqq \frac{1}{(1-t)(1+t)^{n+1}}.
\end{align*}
A direct calculation gives
\begin{align*}
  A_n(t)-\frac12A_{n-1}(t) & =  \frac{1-t^{n-1}}{2(1+t)^{n+1}},\\
  B_n(t)-\frac12B_{n-1}(t) & =  \frac{1}{2(1+t)^{n+1}}.
\end{align*}
Substituting these two identities into the one-dimensional representation of $\mathcal F_n-\mathcal F_{n-1}/2$ in Lemma~\ref{lem:1DIntegralRepresneation} proves the first equality.
To verify the second identity of the lemma, we abbreviate
\begin{align*}
  I_n\coloneqq \int_0^1 \frac{(1-t^{n-1})\Phi(t)+\log^2(t)}{(1+t)^{n+1}} \dt.
\end{align*}
Since $\operatorname{Li}_2(0) = 0$ and $\operatorname{Li}_2(-1) = - \pi^2/12$ \cite[Thm.~2.6.3]{AndrewsAskeyRoy99}, we have $\Phi(0) = \pi^2/6$ and $\Phi(1) = 0$. Moreover, the identity $\tfrac{\mathrm d}{\mathrm dt} \operatorname{Li}_2(t) = - \log(1-t)/t$  \cite[Eq.~2.6.2]{AndrewsAskeyRoy99} yields $ \Phi'(t) =- 2\log(1+t)/t$ for all $t\in (0,1)$. 
Hence, an integration by parts yields
\begin{align}\label{eq:Proofasdsacvvfq}
\begin{aligned}
  \int_0^1\frac{1-t^{n-1}}{(1+t)^{n+1}} \Phi(t)\dt&  =-\frac{1}{n} \int_0^1 \left(  \frac{1+t^{n}}{(1+t)^{n}} \right)' \Phi(t)\dt\\
  &  = \frac{\pi^2}{6n} - \frac{2}{n}\int_0^1\frac{1+t^n}{(1+t)^n}\frac{\log(1+t)}{t}\dt.
\end{aligned}
\end{align}
To evaluate the second addend in $I_n$, we set $Q_n(t) \coloneqq 1 - (1+t)^{-n}$ for all $t\in (0,1)$.
Its derivative reads $Q_n'(t)=n(1+t)^{-n-1}$. Moreover, we have $Q_n(t)\log^2(t) \to 0$ as $t \to 0$ and $\log^2(1)=0$. Hence, integration by parts yields
\begin{align*}
  n\int_0^1\frac{\log^2 t}{(1+t)^{n+1}}\dt &= \int_0^1 Q_n'(t)\log^2(t)\dt = -2\int_0^1\left(1-\frac1{(1+t)^n}\right)\frac{\log(t)}{t}\dt.
\end{align*}
Combining this identity with \eqref{eq:Proofasdsacvvfq} gives
\begin{align}\label{eq:proof-one-step-In-reduction}
  nI_n-\frac{\pi^2}{6} = -2\int_0^1
  \left[\left(1-\frac1{(1+t)^n}\right)\log(t) + \frac{1+t^n}{(1+t)^n}\log(1+t) \right]\frac{\dt}{t}.
\end{align}
Set $x\coloneqq t/(1+t)$ for all $t\in (0,1)$. Then we have $t=x/(1-x)$, $\dt/t=\dx/(x(1-x))$, and $t\in(0,1)$ corresponds to $x\in(0,1/2)$. Moreover, we obtain
\begin{align*}
\begin{aligned}
  1-\frac1{(1+t)^n}&=1-(1-x)^n,& \qquad \frac{1+t^n}{(1+t)^n}&=x^n+(1-x)^n, \\
  \log t&=\log x-\log(1-x),&  \qquad   \log(1+t)&=-\log(1-x).
\end{aligned}
\end{align*}
Hence, substituting $x$ in \eqref{eq:proof-one-step-In-reduction} shows
\begin{align}\label{eq:proof-one-step-Mn}
  \frac12\left(nI_n-\frac{\pi^2}{6}\right) = \int_0^{1/2} \frac{-[1-(1-x)^n]\log x+[1+x^n]\log(1-x)}{x(1-x)} \dx.
\end{align}
We define the integral 
\begin{align*}
  S_n \coloneqq \int_0^1 \frac{1-x^{n-1}}{1-x}[-\log(1-x)]\dx.
\end{align*}
Splitting the integral at $1/2$ and substituting $x\mapsto1-x$ on $(1/2,1)$ gives
\begin{align}\label{eq:proof-one-step-Sn}
  S_n = \int_0^{1/2} \frac{-(x-x^n)\log(1-x)-[1-x-(1-x)^n]\log x}{x(1-x)} \dx.
\end{align}
Subtracting \eqref{eq:proof-one-step-Sn} from \eqref{eq:proof-one-step-Mn}, we obtain
\begin{align*}
\frac12\left(nI_n-\frac{\pi^2}{6}\right)-S_n = \int_0^{1/2} \frac{-x\log x+(1+x)\log(1-x)}{x(1-x)}\dx \eqqcolon J.
\end{align*}
By the substitution $x \mapsto (1+e^s)^{-1}$ for all $x\in(0,1/2)$ with  $\dx/(x(1-x))=-\ds$, we obtain for the integral $J$  that
\begin{align*}
  J =\int_0^\infty \left(\frac{s}{1+e^s}-\log(1+e^{-s})\right)\ds .
\end{align*}
Since $\log(1+e^{-s}) = \int_s^\infty \frac{1}{1+e^r}\dr$,
Fubini's theorem gives
\begin{align*}
  \int_0^\infty \log(1+e^{-s})\ds = \int_0^\infty \int_s^\infty \frac{1}{1+e^r}\dr\ds = \int_0^\infty \frac{r}{1+e^r}\dr.
\end{align*}
Hence, both terms cancel, resulting in $J=0$ and consequently
\begin{align*}
  \frac12\left(nI_n-\frac{\pi^2}{6}\right)=S_n.
\end{align*}
It remains to evaluate $S_n$. The identity
\begin{align*}
  \frac{1-x^{n-1}}{1-x} = \sum_{\ell=0}^{n-2}x^\ell\qquad\text{yields}\qquad
  S_n = \sum_{\ell=0}^{n-2}\int_0^1 x^\ell[-\log(1-x)]\dx.
\end{align*}
Using $-\log(1-x)=\sum_{m=1}^\infty x^m/m$ and monotone convergence, we obtain
\begin{align*}
  \int_0^1 x^\ell[-\log(1-x)]\dx
  &= \sum_{m=1}^\infty \frac1m\int_0^1 x^{m+\ell}\dx = \sum_{m=1}^\infty\frac{1}{m(m+\ell+1)} \\
  &= \frac1{\ell+1}\sum_{m=1}^\infty \left(\frac1m-\frac1{m+\ell+1}\right)
  = \frac{H_{\ell+1}}{\ell+1}.
\end{align*}
This yields the identity
\begin{align*}
  S_n = \sum_{\ell=0}^{n-2}\frac{H_{\ell+1}}{\ell+1} = \sum_{j=1}^{n-1}\frac{H_j}{j}\qquad \text{and consequently}\qquad   I_n = \frac{\pi^2}{6n} + \frac{2}{n}\sum_{j=1}^{n-1}\frac{H_j}{j}.
\end{align*}
Together with the first equality of the lemma this proves
\begin{align*}
  \mathcal F_n-\frac12\mathcal F_{n-1} & = 3I_n = \frac{\pi^2}{2n} + \frac{6}{n}\sum_{j=1}^{n-1}\frac{H_j}{j}.\qedhere
\end{align*}
\end{proof}
We use the lemma's iterative description to conclude a closed formula, involving the numbers 
\begin{align*}
\zeta(3) \coloneqq \sum_{j=1}^\infty \frac{1}{j^3},\qquad 
  H_m\coloneqq \sum_{j=1}^m\frac1j\qquad\text{and} \qquad
  H_m^{(2)}\coloneqq \sum_{j=1}^m\frac1{j^2}\qquad\text{for }m\in \mathbb{N}_0.
\end{align*}
\begin{theorem}[Formula for $\mathcal{F}_n$]\label{thm:compact-one-axis-moments}
We have for all $n\in\mathbb N_0$ the identity
\begin{align*}
  \mathcal F_n = \frac{21}{2^n}\zeta(3) + \frac{\pi^2}{2^n}\sum_{m=1}^n\frac{2^{m-1}}{m} + \frac{3}{2^n}\sum_{m=1}^n\frac{2^m}{m}\left(H_{m-1}^2+H_{m-1}^{(2)}\right).
\end{align*}
\end{theorem}
\begin{proof}
We first consider the case $n=0$. The one-dimensional representation established in Lemma~\ref{lem:1DIntegralRepresneation} and the decomposition $(1-t)^{-1}(1+t)^{-1}
=\big((1-t)^{-1}+(1+t)^{-1}\big)/2$ for all $t\in (0,1)$ gives
\begin{align*}
  \mathcal F_0 &= 12\int_0^1\frac{\Phi(t)}{(1-t)(1+t)}\dt +6\int_0^1\frac{\log^2(t)}{(1-t)(1+t)}\dt \\
   & = 6 \int_0^1 \frac{\Phi(t)}{1-t} \dt + 6 \int_0^1 \frac{\Phi(t)}{1+t} \dt + 3\int_0^1 \frac{\log^2(t)}{1-t} \dt + 3\int_0^1 \frac{\log^2(t)}{1+t} \dt.
\end{align*}
By the series expansion $(1-t)^{-1} = \sum_{j=0}^\infty t^j$ for $t\in (0,1)$ we have
\begin{align}\label{eq:Proofoiiqi1}
 \int_0^1\frac{\log^2 t}{1-t}\,\mathrm dt = \sum_{j=0}^\infty \int_0^1 t^j\log^2 t\,\mathrm dt =  2\sum_{j=0}^\infty\frac1{(j+1)^3} = 2\zeta(3).
\end{align}
Similarly, using $(1+t)^{-1} = \sum_{j=0}^\infty (-1)^j t^j$ for $t\in (0,1)$ and the alternating zeta series/Dirichlet eta function \cite[Eq.~23.2.19]{AbramowitzStegun64}, we get
\begin{align}\label{eq:Proofoiiqi2}
\begin{aligned}
 \int_0^1\frac{\log^2(t)}{1+t}\dt & =  \sum_{j=0}^\infty (-1)^j \int_0^1 t^j\log^2 t\,\mathrm dt = 2\sum_{j=0}^\infty\frac{(-1)^j}{(j+1)^3} = 2\sum_{r=1}^\infty\frac{(-1)^{r-1}}{r^3} \\
  & = 2\left(1-2^{-2}\right)\zeta(3) = \frac32\zeta(3).
\end{aligned}
\end{align}
Since $\operatorname{Li}_2(-1) = - \pi^2/12$ \cite[Thm.~2.6.3]{AndrewsAskeyRoy99}, we have $\Phi(1) = 0$. Moreover, the characterization of $\operatorname{Li}_2$ as an integral in  \cite[Eq.~2.6.2]{AndrewsAskeyRoy99} yields $ \Phi'(t) = -2\log(1+t)/t$ for all $t\in (0,1)$. 
This shows that $\Phi(t)$ tends to zero as $t\nearrow 1$ with a linear rate, showing $\lim_{t\nearrow 1}\log(1-t)\Phi(t)=0$. Hence, integration by parts leads to
\begin{align*}
  \int_0^1\frac{\Phi(t)}{1-t}\,\mathrm dt & = - \int_0^1 \Phi(t) (\log(1-t))' \dt = -2\int_0^1 \frac{\log(1-t)\log(1+t)}{t} \dt.
\end{align*}
Expanding $\log(1+t)/t = \sum_{r=1}^{\infty}(-1)^{r-1}\frac{t^{r-1}}{r}$ for all $t\in (0,1)$ and using $\int_0^1 t^{r-1}\log(1-t)\,\mathrm dt = -{H_r}/{r}$, we obtain
\begin{align*}
 -2 \int_0^1 \frac{\log(1-t)\log(1+t)}{t} \dt  = 2 \sum_{r=1}^{\infty}\frac{(-1)^{r-1}H_r}{r^2}.
\end{align*}
A computation of the alternating Euler-sum can be found in \cite[Thm.~4.1]{Sitaramachandra87}, where the sum above is denoted by $H_7(2)$. The resulting value reads
\begin{align}\label{eq:Proofoiiqi3}
2  \sum_{r=1}^{\infty}\frac{(-1)^{r-1}H_r}{r^2} = \frac54\zeta(3).
\end{align}
Likewise, integrating by parts gives
\begin{align*}
  \int_0^1\frac{\Phi(t)}{1+t}\dt =   \int_0^1 \Phi(t) (\log(1+t))'\dt = 2\int_0^1\frac{\log^2(1+t)}{t}\dt.
\end{align*}
To evaluate the last integral, we use $\log(1+t) = \sum_{m=1}^{\infty}(-1)^{m-1} t^m/m$ and conclude
\begin{align*}
  \log^2(1+t) = 2\sum_{r=1}^{\infty}(-1)^r\frac{H_{r-1}}{r}t^r.
\end{align*}
Integrating the addends, using $H_{r-1}=H_r-r^{-1}$, and exploiting the identities in \eqref{eq:Proofoiiqi2} and \eqref{eq:Proofoiiqi3} yield
\begin{align*}
  \int_0^1\frac{\log^2(1+t)}{t}\dt & = 2\sum_{r=1}^{\infty}(-1)^r\frac{H_{r-1}}{r^2} = -2\sum_{r=1}^{\infty}\frac{(-1)^{r-1}H_r}{r^2} + 2\sum_{r=1}^{\infty}\frac{(-1)^{r-1}}{r^3}\\
  & = - \frac{5}{4}\zeta(3) + \frac{3}{2}\zeta(3) = \frac14 \zeta(3).
\end{align*}
Combining the identities results in 
\begin{align}\label{eq:Proofoiiqi4}
 \int_0^1\frac{\Phi(t)}{1+t}\dt = \frac12 \zeta(3).
\end{align}
The values in \eqref{eq:Proofoiiqi1}--\eqref{eq:Proofoiiqi4} yield the asserted formula for $n=0$; that is,
\begin{align*}
  \mathcal F_0  &=6\left(\frac54+\frac12\right)\zeta(3) +3\left(2+\frac32\right)\zeta(3) = 21 \zeta(3).
\end{align*}

We proceed with the case $n\in\mathbb N$. Iterating the formula in Lemma~\ref{lem:one-axis-one-step-identity} yields
\begin{align*}
  \mathcal F_n = 2^{-n}\mathcal F_0 + \sum_{m=1}^n2^{m-n} \left(\frac{\pi^2}{2m} + \frac{6}{m}\sum_{j=1}^{m-1}\frac{H_j}{j} \right).
\end{align*}
Using $\mathcal F_0=21\zeta(3)$, this becomes
\begin{align*}
  \mathcal F_n = \frac{21}{2^n}\zeta(3)  + \frac{\pi^2}{2^n}\sum_{m=1}^n\frac{2^{m-1}}{m} + \frac{6}{2^n}\sum_{m=1}^n\frac{2^m}{m}\sum_{j=1}^{m-1}\frac{H_j}{j}.
\end{align*}
To simplify the formula, we use the identity \cite[Eq.~3]{Coppo21}
\begin{align*}
  \sum_{j=1}^{m-1}\frac{H_j}{j} = \frac12\left(H_{m-1}^2+H_{m-1}^{(2)}\right).
\end{align*}
Combining these identities concludes the proof.
\end{proof}

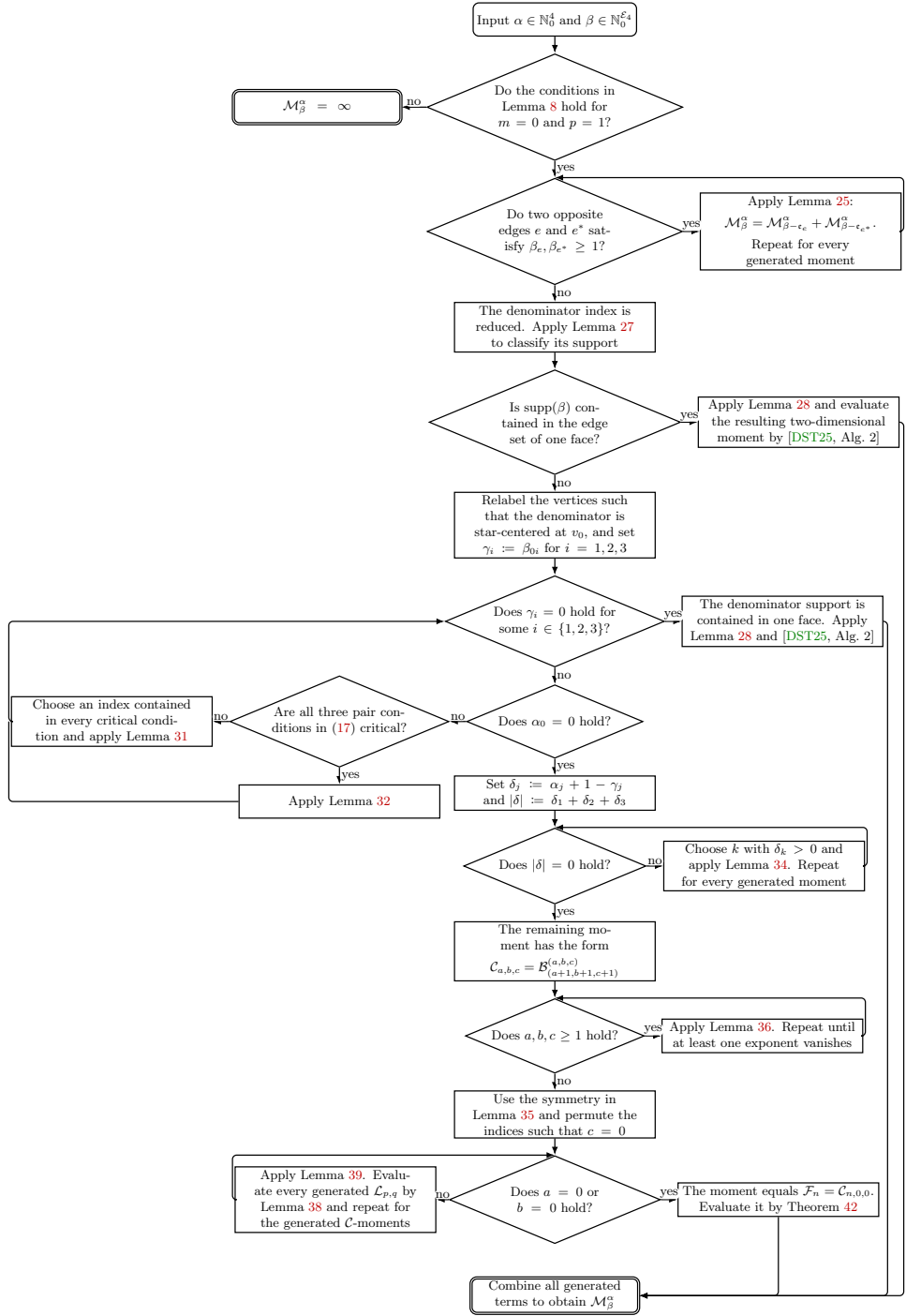
\begin{figure}
\centering
\begin{adjustbox}{max width=\textwidth,center}
\begin{tikzpicture}[
  font=\small,
  node distance=4mm and 3mm,
  >={Latex[length=2mm]},
  line/.style={
    draw,
    ->,
    thick,
    rounded corners=1mm
  },
  merge/.style={
    draw,
    thick,
    rounded corners=1mm
  },
  startstop/.style={
    rectangle,
    rounded corners,
    draw,
    thick,
    align=center,
    inner sep=2pt,
    minimum height=8mm,
    text width=38mm
  },
  process/.style={
    rectangle,
    draw,
    thick,
    align=center,
    inner sep=2pt,
    minimum height=8mm,
    text width=47mm
  },
  decision/.style={
    diamond,
    draw,
    thick,
    align=center,
    aspect=2.4,
    inner sep=1pt,
    text width=34mm
  },
  terminal/.style={
    rectangle,
    rounded corners,
    draw,
    double,
    thick,
    align=center,
    inner sep=2pt,
    minimum height=8mm,
    text width=39mm
  },
  label/.style={
    midway,
    fill=white,
    inner sep=1pt
  }
]

\node[startstop] (input)
  {Input $\alpha\in\mathbb N_0^4$ and $\beta\in\mathbb N_0^{\mathcal E_4}$};

\node[decision, below=of input] (integrable)
  {Do the conditions in Lemma~\ref{lem:Regularity} hold for $m=0$ and $p=1$?};

\node[terminal, left=6mm of integrable] (divergent)
  {$\mathcal M_\beta^\alpha = \infty$};

\node[decision, below=of integrable] (opposite)
  {Do two opposite edges $e$ and $e^\ast$ satisfy $\beta_e,\beta_{e^\ast}\geq1$?};

\node[process, right=4mm of opposite] (opposite-reduction)
  {Apply Lemma~\ref{lem:opposite-edge-reduction}:%
  \begin{align*}
    \mathcal M_\beta^\alpha
    =\mathcal M_{\beta-\fre_e}^\alpha
    +\mathcal M_{\beta-\fre_{e^\ast}}^\alpha.
  \end{align*}%
  Repeat for every generated moment};

\node[process, below=of opposite] (classification)
  {The denominator index is reduced. Apply Lemma~\ref{lem:reduced-index} to classify its support};

\node[decision, below=of classification] (face)
  {Is $\operatorname{supp}(\beta)$ contained in the edge set of one face?};

\node[process, right=4mm of face] (face-evaluation)
  {Apply Lemma~\ref{lem:triangular-moments} and evaluate the resulting two-dimensional moment by \cite[Alg.~2]{DieningStornTscherpel24}};

\node[process, below=of face] (star)
  {Relabel the vertices such that the denominator is star-centered at $v_0$, and set $\gamma_i\coloneqq\beta_{0i}$ for $i=1,2,3$};

\node[decision, below=of star] (zero-gamma)
  {Does $\gamma_i=0$ hold for some $i\in\lbrace1,2,3\rbrace$?};

\node[process, right=4mm of zero-gamma] (star-face)
  {The denominator support is contained in one face. Apply Lemma~\ref{lem:triangular-moments} and \cite[Alg.~2]{DieningStornTscherpel24}};

\node[decision, below=of zero-gamma] (alpha-zero)
  {Does $\alpha_0=0$ hold?};

\node[decision, left=4mm of alpha-zero] (critical)
  {Are all three pair conditions in \eqref{eq:star-pair-criterion} critical?};

\node[process, left=4mm of critical] (simple-star)
  {Choose an index contained in every critical condition and apply Lemma~\ref{lem:simple-admissible-star-reduction}};

\node[process, below=of critical] (critical-star)
  {Apply Lemma~\ref{lem:critical-star-reduction}};

\node[process, below=of alpha-zero] (defect)
  {Set $\delta_j\coloneqq\alpha_j+1-\gamma_j$ and $|\delta|\coloneqq\delta_1+\delta_2+\delta_3$};

\node[decision, below=of defect] (delta-zero)
  {Does $|\delta|=0$ hold?};

\node[process, right=4mm of delta-zero] (boundary-descent)
  {Choose $k$ with $\delta_k>0$ and apply Lemma~\ref{lem:boundary-star-descent}. Repeat for every generated moment};

\node[process, below=of delta-zero] (diagonal)
  {The remaining moment has the form%
  \begin{align*}
    \mathcal C_{a,b,c}
    =\mathcal B_{(a+1,b+1,c+1)}^{(a,b,c)}
  \end{align*}};

\node[decision, below=of diagonal] (all-positive)
  {Does $a,b,c\geq1$ hold?};

\node[process, right=4mm of all-positive] (cube-reduction)
  {Apply Lemma~\ref{lem:cube-reduction-one-zero-exponent}. Repeat until at least one exponent vanishes};

\node[process, below=of all-positive] (symmetry)
  {Use the symmetry in Lemma~\ref{lem:diagonal-boundary-star-cube} and permute the indices such that $c=0$};

\node[decision, below=of symmetry] (one-axis)
  {Does $a=0$ or $b=0$ hold?};

\node[process, right=4mm of one-axis] (one-axis-value)
  {The moment equals $\mathcal F_n=\mathcal C_{n,0,0}$. Evaluate it by Theorem~\ref{thm:compact-one-axis-moments}};

\node[process, left=3mm of one-axis] (two-index)
  {Apply Lemma~\ref{lem:two-index-reduction-to-Fn}. Evaluate every generated $\mathcal L_{p,q}$ by Lemma~\ref{lem:explicit-logarithmic-moments} and repeat for the generated $\mathcal C$-moments};

\node[terminal, below=8mm of one-axis] (result)
  {Combine all generated terms to obtain $\mathcal M_\beta^\alpha$};

\draw[line]
  (input)
  --
  (integrable);

\draw[line]
  (integrable)
  --
  node[label,above] {no}
  (divergent);

\draw[line]
  (integrable)
  --
  node[label,right] {yes}
  (opposite);

\draw[line]
  (opposite)
  --
  node[label,above] {yes}
  (opposite-reduction);

\draw[line]
  (opposite-reduction.east)
  --
  ++(0.5mm,0)
  |-
  (opposite.north);

\draw[line]
  (opposite)
  --
  node[label,right] {no}
  (classification);

\draw[line]
  (classification)
  --
  (face);

\draw[line]
  (face)
  --
  node[label,above] {yes}
  (face-evaluation);

\draw[line]
  (face)
  --
  node[label,right] {no}
  (star);

\draw[line]
  (star)
  --
  (zero-gamma);

\draw[line]
  (zero-gamma)
  --
  node[label,above] {yes}
  (star-face);

\draw[line]
  (zero-gamma)
  --
  node[label,right] {no}
  (alpha-zero);

\draw[line]
  (alpha-zero)
  --
  node[label,above] {no}
  (critical);

\draw[line]
  (critical)
  --
  node[label,above] {no}
  (simple-star);

\draw[line]
  (critical)
  --
  node[label,right] {yes}
  (critical-star);

\coordinate (star-return) at ($(simple-star.west)+(-0.5mm,0)$);

\draw[line]
  (simple-star.west)
  --
  (star-return)
  |-
  (zero-gamma.west);

\draw[merge]
  (critical-star.west)
  -|
  (star-return);

\draw[line]
  (alpha-zero)
  --
  node[label,right] {yes}
  (defect);

\draw[line]
  (defect)
  --
  (delta-zero);

\draw[line]
  (delta-zero)
  --
  node[label,above] {no}
  (boundary-descent);

\draw[line]
  (boundary-descent.east)
  --
  ++(0.5mm,0)
  |-
  (delta-zero.north);

\draw[line]
  (delta-zero)
  --
  node[label,right] {yes}
  (diagonal);

\draw[line]
  (diagonal)
  --
  (all-positive);

\draw[line]
  (all-positive)
  --
  node[label,above] {yes}
  (cube-reduction);

\draw[line]
  (cube-reduction.east)
  --
  ++(0.5mm,0)
  |-
  (all-positive.north);

\draw[line]
  (all-positive)
  --
  node[label,right] {no}
  (symmetry);

\draw[line]
  (symmetry)
  --
  (one-axis);

\draw[line]
  (one-axis)
  --
  node[label,above] {yes}
  (one-axis-value);

\draw[line]
  (one-axis)
  --
  node[label,above] {no}
  (two-index);

\draw[line]
  (two-index.west)
  --
  ++(-0.5mm,0)
  |-
  (one-axis.north);

\draw[line]
  (one-axis-value.south)
  |-
  (result.east);

\draw[line]
  (face-evaluation.east)
  --
  ++(1mm,0)
  |-
  (result.east);

\draw[line]
  (star-face.east)
  --
  ++(1mm,0)
  |-
  (result.east);

\end{tikzpicture}
\end{adjustbox}
\caption{Flowchart for the evaluation of the rational moment $\mathcal M_\beta^\alpha$. The recursive reductions terminate because they successively decrease $|\beta|$, $\alpha_0$, $|\delta|$, $a+b+c$, or $a+b$.}
\label{fig:rational-moment-flowchart}
\end{figure}
\section{Numerical scheme}
Before we conclude this paper with numerical experiments in Section~\ref{subsec:NumExp}, we briefly discuss our implementation.  It is inspired by the 2D approach discussed in \cite{DieningStornTscherpel24}.
However, we use a more efficient assembly of the stiffness matrix, discussed in the following subsection.

\subsection{Assembly of the stiffness matrix}\label{subsec:efficient-stiffness-assembly}
The two-dimensional implementation in \cite[Sec.~4.2.1]{DieningStornTscherpel24} suggests a direct strategy for the assembly of the stiffness matrix: expand the local basis functions into geometry-independent rational functions, compute all products of their second barycentric derivatives, and evaluate the resulting rational moments once in an offline phase. For our three-dimensional problem this approach reads as follows. 
Let $(\psi_a)_{a=1}^N$ be a fixed family of geometry-independent rational functions such that every local basis function can be represented as
\begin{align*}
\varphi_r^T\circ F_T=\sum_{a=1}^N c_{ar}(T)\psi_a\qquad\text{with affine transformation }F_T \colon \widehat{T}\to T.
\end{align*}
For the singular Zienkiewicz element considered in Section~\ref{subbsec:SingZienk}, our representation uses $N=64$ geometry-independent rational functions for the $28$ local basis functions of $Z_3^{\operatorname{sing}}(T)$. Using the exact differentiation, multiplication, and integration formulae discussed in Lemma~\ref{lem:Multi}--\ref{lem:Differentiatoin} and Section~\ref{sec:3d-rational-moments}, a direct extension of the 2D routine then computes in an offline phase for all $a,b = 1,\dots,64$ and $i,j,k,\ell = 0,\dots,3$
\begin{align*}
\widehat{A}(a,b,i,j,k,\ell) \coloneqq \dashint_{\widehat T} \partial_{\lambda_i}\partial_{\lambda_j} \psi_a \partial_{\lambda_k}\partial_{\lambda_\ell} \psi_b \,\mathrm d\widehat x.
\end{align*}
This tensor in $\mathbb{R}^{64\times 64 \times 4\times 4 \times 4\times 4}$ is then used to compute the local stiffness matrix via Frobenius inner products with a tensor $Q_T \in \mathbb{R}^{4\times 4\times 4\times 4}$ depending on $T$. Each of these inner products requires $2\cdot 4^4 = 512$ floating-point operations, resulting in the total cost of
\begin{align}\label{eq:BadPerformance}
64^2 \cdot 512 = 2097152 \approx 2\cdot 10^6\quad \text{floating-point operations per tetrahedron}.
\end{align} 
To improve the performance in \eqref{eq:BadPerformance}, we exploit the following strategy which differs from \cite{DieningStornTscherpel24}. Let $F_T\colon\widehat T\to T = [v_0,\dots,v_3]$ be the affine map
\begin{align*}
F_T(\widehat x)\coloneqq v_0+B_T\widehat x \qquad\text{with}\qquad B_T\coloneqq\big(v_1-v_0,v_2-v_0,v_3-v_0\big).
\end{align*}
We use the three independent reference coordinates
\begin{align*}
\widehat x=(\lambda_1,\lambda_2,\lambda_3)^\top \qquad\text{with}\qquad \lambda_0=1-\lambda_1-\lambda_2-\lambda_3.
\end{align*}
For a function $\psi$ in these coordinates, the affine transformation of the Hessian reads
\begin{align*}
D_x^2(\psi\circ F_T^{-1})=B_T^{-\top}D_{\widehat x}^2\psi\,B_T^{-1}.
\end{align*}
Consequently, for two reference Hessians $H_1,H_2\in\mathbb R^{3\times3}_{\operatorname{sym}}$, we obtain
\begin{align*}
(B_T^{-\top}H_1B_T^{-1}):(B_T^{-\top}H_2B_T^{-1})=\operatorname{tr}(H_1G_TH_2G_T) \ \text{ with }\ G_T\coloneqq B_T^{-1}B_T^{-\top}\in\mathbb R^{3\times3}_{\operatorname{sym}}.
\end{align*}
The symmetric matrix $G_T$ has only six independent entries. Let $(E_\mu)_{\mu=1}^6$ be a fixed basis of $\mathbb R^{3\times3}_{\operatorname{sym}}$ and write
\begin{align*}
G_T=\sum_{\mu=1}^6 g_\mu(T)E_\mu.
\end{align*}
Inserting this representation into the Hessian product gives
\begin{align*}
&\operatorname{tr}(H_1G_TH_2G_T) = \sum_{\mu,\nu=1}^6 g_\mu(T) g_\nu(T) \operatorname{tr}(H_1E_\mu H_2E_\nu)\\
& = \sum_{\mu=1}^6 g_\mu(T)^2 \operatorname{tr}(H_1E_\mu H_2E_\mu) + \sum_{1 \leq \mu <\nu\leq 6} g_\mu(T) g_\nu(T) \operatorname{tr}(H_1E_\mu H_2E_\nu+H_1E_\nu H_2E_\mu).
\end{align*}
Thus, only $6+15 = 21$ geometry coefficients are required.
More explicitly, define
\begin{align*}
\gamma_{\mu\nu}(T)\coloneqq g_\mu(T)g_\nu(T) \qquad\text{for all }1 \leq \mu \leq \nu \leq 6.
\end{align*}
Furthermore, define the geometry-independent bilinear forms
\begin{align*}
\begin{aligned}
\mathfrak b_{\mu\mu}(H_1,H_2)&\coloneqq\operatorname{tr}(H_1E_\mu H_2E_\mu) && \quad\text{for all }\mu = 1,\dots,6,\\
\mathfrak b_{\mu\nu}(H_1,H_2)&\coloneqq\operatorname{tr}(H_1E_\mu H_2E_\nu)+\operatorname{tr}(H_1E_\nu H_2E_\mu) &&\quad \text{for all }1 \leq \mu < \nu \leq 6.
\end{aligned}
\end{align*}
This yields the identity
\begin{align*}
\operatorname{tr}(H_1G_TH_2G_T)=\sum_{1\leq\mu\leq\nu\leq6}\gamma_{\mu\nu}(T)\mathfrak b_{\mu\nu}(H_1,H_2).
\end{align*}
We thus precompute for all $1\leq\mu \leq \nu\leq 6$ the matrices $M_{\mu\nu}\in\mathbb R^{64\times 64}$ with entries
\begin{align*}
(M_{\mu\nu})_{ab}\coloneqq\dashint_{\widehat T}\mathfrak b_{\mu\nu}\big(D_{\widehat x}^2\psi_a,D_{\widehat x}^2\psi_b\big)\, \mathrm d \widehat x \qquad\text{for all } a,b = 1,\dots,64.
\end{align*}
Combining these identities yields the local stiffness matrix
\begin{align}\label{eq:efficient-local-stiffness}
A_T = |T|\, C_T^\top\Big(\sum_{1\leq\mu\leq\nu\leq 6}\gamma_{\mu\nu}(T)M_{\mu\nu}\Big)C_T
\end{align}
with
\begin{align*}
(A_T)_{r,s} = \int_T D^2 \varphi_r : D^2 \varphi_s \dx\qquad\text{for all }r,s=1,\dots,28.
\end{align*}
The computation of the inner sum thus consists of $2\cdot 21 \cdot 64^2 = 172032$ floating-point operations and thus improves significantly on \eqref{eq:BadPerformance}. Moreover, storing the matrices $M_{\mu \nu}$ for all $1\leq \mu \leq \nu \leq 6$ requires less memory.

\subsection{Numerical experiments}\label{subsec:NumExp}
We conclude this paper with two numerical experiments. 
In both experiments, we compare our routine to results obtained by numerical quadrature. 
We use the three-dimensional analogue of the numerical quadrature employed in \cite[Sec.~5.1]{DieningStornTscherpel24}. More precisely, set the reference tetrahedron
\begin{align*}
\widehat T \coloneqq \lbrace (x,y,z)\in\mathbb R^3\colon x,y,z \geq 0,\ x+y+z \leq 1\rbrace.
\end{align*}
Fubini's theorem yields for $g \colon \widehat{T}\to \mathbb{R}$
\begin{align*}
\int_{\widehat T}g(x,y,z)\,\mathrm d(x,y,z) = \int_0^1\int_0^{1-x}\int_0^{1-x-y}g(x,y,z)\dz\dy\dx.
\end{align*}
We approximate each of the three one-dimensional integrals successively by an $n$-point Gau\ss{} quadrature rule on the corresponding interval. The resulting routine uses $n^3$ quadrature points.
This approach is exact for all $g\in\mathcal P_{2n-3}(\widehat T)$ and extends via affine transformations to arbitrary tetrahedra $T \in \tria$.

\subsubsection{Quadrature error}
In our first experiment, we investigate the quadrature error for the rational functions $R \coloneqq R_\beta^\alpha$ and $R' \coloneqq R_{\beta'}^{\alpha'}$ with indices $\alpha = (0,0,0,0)$ and $\alpha' = (0,2,3,0)$ as well as 
\begin{align*}
\beta_{ij} = \begin{cases}
1 &\text{for }ij \in \lbrace 01,02,03\rbrace,\\
0&\text{else}
\end{cases}\qquad\text{and}\qquad \beta'_{ij} = \begin{cases}
3 &\text{for }ij = 02,\\
1 & \text{for }ij = 12,\\
0&\text{else}.
\end{cases}
\end{align*}
We use the quadrature rule discussed above for various numbers $n$ of 1D Gau\ss{} points to approximate the moments $\mathcal{M} \coloneqq \mathcal{M}_\beta^\alpha$ and $\mathcal{M}' \coloneqq \mathcal{M}_{\beta'}^{\alpha'}$.
The resulting errors are displayed in Figure~\ref{fig:Exp1}.
It shows that numerical quadrature for the first integrand leads to poor results. This is to be expected, as the regularity properties stated in Lemma~\ref{lem:Regularity} yield $R \in L^{p}(T)$ if and only if $p< 3/2$. Note that consequently it cannot occur in second derivatives of functions in $Z_3^{\operatorname{sing}}(T) \subset W^{2,\infty}(T)$. The second function $R' =  \partial_{\lambda_0}^2 E^2_{01}/2 \in L^\infty(T)$ however does occur. Here we observe better approximation properties: we need about $n=30$ Gau\ss{} points in 1D, leading to $30^3 = 27\, 000$ point evaluations in total, to obtain a relative error of $10^{-5}$. Increasing $n$ leads to better results, however, the speed of convergence seems to reduce.
The impact of these quadrature errors in the approximation of the biharmonic eigenvalue problem is investigated in the following.

\begin{figure}
\begin{tikzpicture}
\begin{axis}[
clip=false,
width=.8\textwidth,
height=.45\textwidth,
ymode = log,
xlabel = {Number of 1D Gau\ss{} points $n$},
cycle multi list={\nextlist MyColors2},
scale = {1},
clip = true,
legend cell align=left,
legend style={legend columns=1,legend pos= south west,font=\fontsize{7}{5}\selectfont}
]
	\addplot table [col sep=comma,x expr=\thisrow{n},y=relativeError] {Experiments/Exp1a.dat};
	\addplot table [col sep=comma,x expr=\thisrow{n},y=relativeError] {Experiments/Exp1a_b.dat};
	
	\legend{{relat.~error $\mathcal{M}$},{relat.~error $\mathcal{M}'$}};
\end{axis}
\end{tikzpicture}

\caption{Convergence history plot of the relative error for the approximation of the moments via the scheme described at the beginning of Section~\ref{subsec:NumExp}.}\label{fig:Exp1}
\end{figure}
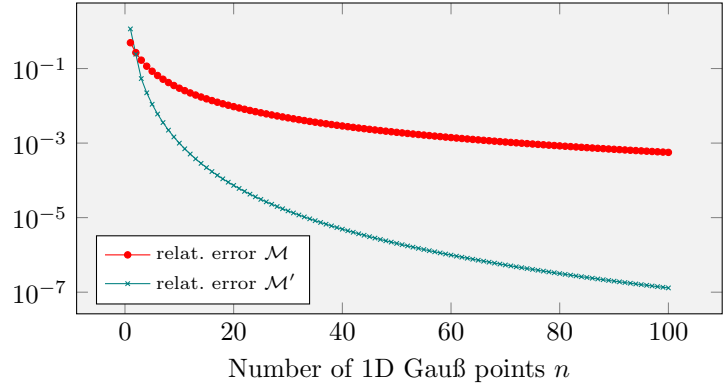

\subsubsection{Eigenvalue problem}
In our second experiment, we investigate the effect of numerical quadrature on the approximation of the smallest biharmonic eigenvalue $\lambda_1>0$. More precisely, we consider the problem of finding $u\in H_0^2(\Omega)\setminus\lbrace0\rbrace$ with
\begin{align*}
\Delta^2u=\lambda_1u\qquad\text{in }\Omega=(0,1)^3.
\end{align*}
We discretize this problem using the singular Zienkiewicz element introduced in Section~\ref{subbsec:SingZienk}. The stiffness and mass matrices are assembled either by exact integration or by the numerical quadrature described above. Figure~\ref{fig:Exp2} displays the relative distance between the eigenvalue obtained by numerical quadrature and the corresponding exactly integrated discrete eigenvalue.
The experiment shows a clear dependence on the number $n$ of one-dimensional Gau\ss{} points. For $n=1$, the quadrature error does not decrease under mesh refinement.
For $n=2$, the quadrature-induced eigenvalue error decreases under mesh refinement, but considerably more slowly than for larger values of $n$.
For $n=3$ a deterioration becomes apparent only on the finer meshes. For larger values of $n$, the observed errors exhibit a similar decay over the considered refinement levels. This behavior may still be pre-asymptotic; in particular, the experiment does not exclude that for fixed quadrature order the error eventually reaches a quadrature-induced plateau.

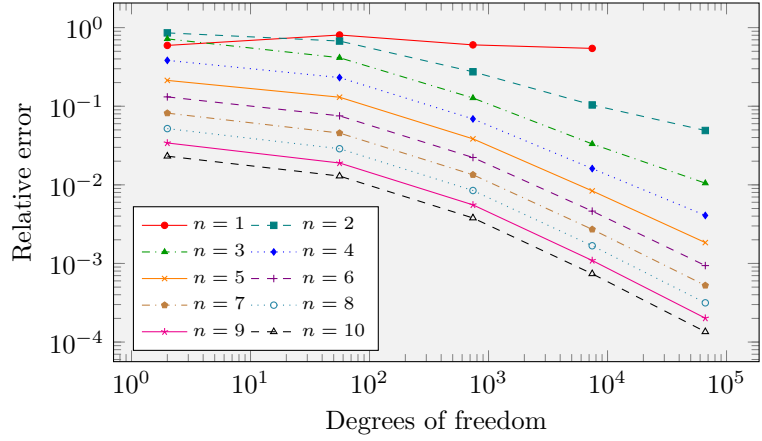
\begin{figure}
\centering
\begin{tikzpicture}
\begin{axis}[
  clip=false,
  width=.8\textwidth,
  height=.5\textwidth,
  xmode=log,
  ymode=log,
  xlabel={Degrees of freedom},
  ylabel={Relative error},
  cycle multi list={\nextlist MyColorsExp2},
  scale={1},
  clip=true,
  unbounded coords=jump,
  legend cell align=left,
  legend style={
    legend columns=2,
    legend pos=south west,
    font=\fontsize{7}{5}\selectfont
  }
]

  \addplot table [
    col sep=comma,
    x=freeDOFs,
    y=relativeDistance,
    restrict expr to domain={\thisrow{nGauss}}{1:1}
  ] {Experiments/Exp2_new.dat};

  \addplot table [
    col sep=comma,
    x=freeDOFs,
    y=relativeDistance,
    restrict expr to domain={\thisrow{nGauss}}{2:2}
  ] {Experiments/Exp2_new.dat};

  \addplot table [
    col sep=comma,
    x=freeDOFs,
    y=relativeDistance,
    restrict expr to domain={\thisrow{nGauss}}{3:3}
  ] {Experiments/Exp2_new.dat};

  \addplot table [
    col sep=comma,
    x=freeDOFs,
    y=relativeDistance,
    restrict expr to domain={\thisrow{nGauss}}{4:4}
  ] {Experiments/Exp2_new.dat};

  \addplot table [
    col sep=comma,
    x=freeDOFs,
    y=relativeDistance,
    restrict expr to domain={\thisrow{nGauss}}{5:5}
  ] {Experiments/Exp2_new.dat};

  \addplot table [
    col sep=comma,
    x=freeDOFs,
    y=relativeDistance,
    restrict expr to domain={\thisrow{nGauss}}{6:6}
  ] {Experiments/Exp2_new.dat};

  \addplot table [
    col sep=comma,
    x=freeDOFs,
    y=relativeDistance,
    restrict expr to domain={\thisrow{nGauss}}{7:7}
  ] {Experiments/Exp2_new.dat};

  \addplot table [
    col sep=comma,
    x=freeDOFs,
    y=relativeDistance,
    restrict expr to domain={\thisrow{nGauss}}{8:8}
  ] {Experiments/Exp2_new.dat};

  \addplot table [
    col sep=comma,
    x=freeDOFs,
    y=relativeDistance,
    restrict expr to domain={\thisrow{nGauss}}{9:9}
  ] {Experiments/Exp2_new.dat};

  \addplot table [
    col sep=comma,
    x=freeDOFs,
    y=relativeDistance,
    restrict expr to domain={\thisrow{nGauss}}{10:10}
  ] {Experiments/Exp2_new.dat};

  \legend{
    {$n=1$},
    {$n=2$},
    {$n=3$},
    {$n=4$},
    {$n=5$},
    {$n=6$},
    {$n=7$},
    {$n=8$},
    {$n=9$},
    {$n=10$}
  };

\end{axis}
\end{tikzpicture}
\caption{Convergence history of the relative error with respect to the number of degrees of freedom for different numbers $n$ of one-dimensional Gau\ss{} points.}
\label{fig:Exp2}
\end{figure}

\subsection*{Acknowledgment}
The author used ChatGPT-5.6 Sol for language editing and presentation improvements and as an assistant in the implementation of the numerical experiments.

\bibliographystyle{amsalpha}
\bibliography{cr_conf_bib}

@book {AbramowitzStegun64,
    AUTHOR = {Abramowitz, Milton and Stegun, Irene A.},
     TITLE = {Handbook of mathematical functions with formulas, graphs, and
              mathematical tables},
    SERIES = {National Bureau of Standards Applied Mathematics Series},
    VOLUME = {No. 55},
      NOTE = {For sale by the Superintendent of Documents},
 PUBLISHER = {U. S. Government Printing Office, Washington, DC},
      YEAR = {1964},
     PAGES = {xiv+1046},
   MRCLASS = {33.00 (65.05)},
  MRNUMBER = {167642},
MRREVIEWER = {D.\ H.\ Lehmer},
}

@article{Alfeld84,
title = {A trivariate {C}lough—{T}ocher scheme for tetrahedral data},
fjournal = {Computer Aided Geometric Design},
journal = {Comput. Aided Geom. Des.},
volume = {1},
number = {2},
pages = {169-181},
year = {1984},
issn = {0167-8396},
doi = {https://doi.org/10.1016/0167-8396(84)90029-3},
url = {https://www.sciencedirect.com/science/article/pii/0167839684900293},
author = {Peter Alfeld},
}

@article{AinsworthParker24,
 author = {Ainsworth, Mark and Parker, Charles},
 title = {Two and three dimensional {{\(H^2\)}}-conforming finite element approximations without {{\(C^1\)}}-elements},
 fjournal = {Computer Methods in Applied Mechanics and Engineering},
 journal = {Comput. Methods Appl. Mech. Eng.},
 issn = {0045-7825},
 volume = {431},
 pages = {24},
 note = {Id/No 117267},
 year = {2024},
 language = {English},
 doi = {10.1016/j.cma.2024.117267},
 zbMATH = {7917744},
 Zbl = {1550.65239}
}

@book {AndrewsAskeyRoy99,
    AUTHOR = {Andrews, George E. and Askey, Richard and Roy, Ranjan},
     TITLE = {Special functions},
    SERIES = {Encyclopedia of Mathematics and its Applications},
    VOLUME = {71},
 PUBLISHER = {Cambridge University Press, Cambridge},
      YEAR = {1999},
     PAGES = {xvi+664},
      ISBN = {0-521-62321-9; 0-521-78988-5},
   MRCLASS = {33-01 (33-02)},
  MRNUMBER = {1688958},
MRREVIEWER = {Bruce\ C.\ Berndt},
       DOI = {10.1017/CBO9781107325937},
       URL = {https://doi.org/10.1017/CBO9781107325937},
}

@article{ChenHuang24,
 author = {Chen, Long and Huang, Xuehai},
 title = {Finite element de {Rham} and {Stokes} complexes in three dimensions},
 fjournal = {Mathematics of Computation},
 journal = {Math. Comput.},
 issn = {0025-5718},
 volume = {93},
 number = {345},
 pages = {55--110},
 year = {2024},
 language = {English},
 doi = {10.1090/mcom/3859},
 zbMATH = {7753424},
 Zbl = {1530.65159}
}

@Book{Ciarlet2002,
  author    = {Ciarlet, P.},
  publisher = {Society for Industrial and Applied Mathematics (SIAM), Philadelphia, PA},
  title     = {{The Finite Element Method for Elliptic Problems}},
  year      = {2002},
  isbn      = {0-89871-514-8},
  note      = {Reprint of the 1978 original [North-Holland, Amsterdam]},
  series    = {Classics in Applied Mathematics},
  volume    = {40},
  doi       = {10.1137/1.9780898719208},
  pages     = {xxviii+530},
  url       = {https://doi.org/10.1137/1.9780898719208},
}

@article {Coppo21,
    AUTHOR = {Coppo, Marc-Antoine},
     TITLE = {New identities involving {C}auchy numbers, harmonic numbers
              and zeta values},
   JOURNAL = {Results Math.},
  FJOURNAL = {Results in Mathematics},
    VOLUME = {76},
      YEAR = {2021},
    NUMBER = {4},
     PAGES = {Paper No. 189, 10},
      ISSN = {1422-6383,1420-9012},
   MRCLASS = {11B68 (05A19 11M06)},
  MRNUMBER = {4305495},
       DOI = {10.1007/s00025-021-01497-0},
       URL = {https://doi.org/10.1007/s00025-021-01497-0},
}

@article{DieningStornTscherpel24,
      title={Exact Integration for singular Zienkiewicz and Guzman-Neilan Finite Elements with Implementation}, 
      author={Lars Diening and Johannes Storn and Tabea Tscherpel},
	fjournal = {Computers & Mathematics with Applications},
	journal = {Comput. Math. Appl.},
	volume = {191},
	pages = {60-85},
	year = {2025},
	issn = {0898-1221},
	doi = {10.1016/j.camwa.2025.04.019},
	url = {https://www.scienc0edirect.com/science/article/pii/S0898122125001634},
}

@article{FalkNeilan13,
author = {Falk, Richard S. and Neilan, Michael},
title = {Stokes Complexes and the Construction of Stable Finite Elements with Pointwise Mass Conservation},
   JOURNAL = {SIAM J. Numer. Anal.},
volume = {51},
number = {2},
pages = {1308-1326},
year = {2013},
doi = {10.1137/120888132},
URL = {https://doi.org/10.1137/120888132},
eprint = {https://doi.org/10.1137/120888132},
}

@article{FengNeilan14,
 author = {Feng, Xiaobing and Neilan, Michael},
 title = {Finite element approximations of general fully nonlinear second order elliptic partial differential equations based on the vanishing moment method},
 fjournal = {Computers \& Mathematics with Applications},
 journal = {Comput. Math. Appl.},
 issn = {0898-1221},
 volume = {68},
 number = {12},
 pages = {2182--2204},
 year = {2014},
 language = {English},
 doi = {10.1016/j.camwa.2014.07.023},
 zbMATH = {6721995},
 Zbl = {1361.35061}
}

@article {FuGuzmanNeilan20,
    AUTHOR = {Fu, Guosheng and Guzm\'an, Johnny and Neilan, Michael},
     TITLE = {Exact smooth piecewise polynomial sequences on {A}lfeld
              splits},
   JOURNAL = {Math. Comp.},
  FJOURNAL = {Mathematics of Computation},
    VOLUME = {89},
      YEAR = {2020},
    NUMBER = {323},
     PAGES = {1059--1091},
      ISSN = {0025-5718,1088-6842},
   MRCLASS = {65N30},
  MRNUMBER = {4063312},
MRREVIEWER = {Fei\ Wang},
       DOI = {10.1090/mcom/3520},
       URL = {https://doi.org/10.1090/mcom/3520},
}

@article{GallistlTran24,
 author = {Gallistl, Dietmar and Tran, Ngoc Tien},
 title = {Stability and guaranteed error control of approximations to the {Monge}-{Amp{\`e}re} equation},
 fjournal = {Numerische Mathematik},
 journal = {Numer. Math.},
 issn = {0029-599X},
 volume = {156},
 number = {1},
 pages = {107--131},
 year = {2024},
 language = {English},
 doi = {10.1007/s00211-023-01385-5},
 zbMATH = {7802483},
 Zbl = {1532.35250}
}

@article{GallistlTran25,
    author = {Gallistl, Dietmar and Tran, Ngoc Tien},
    title = {Minimal residual discretization of a class of fully nonlinear elliptic PDE},
    journal = {IMA Journal of Numerical Analysis},
    pages = {draf075},
    year = {2025},
    month = {08},
    issn = {0272-4979},
    doi = {10.1093/imanum/draf075},
    url = {https://doi.org/10.1093/imanum/draf075},
    eprint = {https://academic.oup.com/imajna/advance-article-pdf/doi/10.1093/imanum/draf075/64149990/draf075.pdf},
}

@article{GuzmanLischkeNeilan22,
 author = {Guzm{\'a}n, Johnny and Lischke, Anna and Neilan, Michael},
 title = {Exact sequences on {Worsey}-{Farin} splits},
 fjournal = {Mathematics of Computation},
 journal = {Math. Comput.},
 issn = {0025-5718},
 volume = {91},
 number = {338},
 pages = {2571--2608},
 year = {2022},
 language = {English},
 doi = {10.1090/mcom/3746},
 zbMATH = {7582809},
 Zbl = {1501.65120}
}

@article {GuzmanNeilan14,
    AUTHOR = {Guzm\'an, Johnny and Neilan, Michael},
     TITLE = {Conforming and divergence-free {S}tokes elements in three
              dimensions},
   JOURNAL = {IMA J. Numer. Anal.},
  FJOURNAL = {IMA Journal of Numerical Analysis},
    VOLUME = {34},
      YEAR = {2014},
    NUMBER = {4},
     PAGES = {1489--1508},
      ISSN = {0272-4979,1464-3642},
   MRCLASS = {65N30 (65N12 76D07)},
  MRNUMBER = {3269433},
MRREVIEWER = {David\ Maltese},
       DOI = {10.1093/imanum/drt053},
       URL = {https://doi.org/10.1093/imanum/drt053},
}

@article{HuLinWu24,
 author = {Hu, Jun and Lin, Ting and Wu, Qingyu},
 title = {A construction of {{\(C^r\)}} conforming finite element spaces in any dimension},
 fjournal = {Foundations of Computational Mathematics},
 journal = {Found. Comput. Math.},
 issn = {1615-3375},
 volume = {24},
 number = {6},
 pages = {1941--1977},
 year = {2024},
 language = {English},
 doi = {10.1007/s10208-023-09627-6},
 zbMATH = {7963968},
 Zbl = {1555.65161}
}

@article{JohnLinkeMerdonNeilanRebholz17,
 author = {John, Volker and Linke, Alexander and Merdon, Christian and Neilan, Michael and Rebholz, Leo G.},
 title = {On the divergence constraint in mixed finite element methods for incompressible flows},
 fjournal = {SIAM Review},
 journal = {SIAM Rev.},
 issn = {0036-1445},
 volume = {59},
 number = {3},
 pages = {492--544},
 year = {2017},
 language = {English},
 doi = {10.1137/15M1047696},
 zbMATH = {6760418},
 Zbl = {1426.76275}
}

@article {Lasserre98,
    AUTHOR = {Lasserre, Jean B.},
     TITLE = {Integration on a convex polytope},
   JOURNAL = {Proc. Amer. Math. Soc.},
  FJOURNAL = {Proceedings of the American Mathematical Society},
    VOLUME = {126},
      YEAR = {1998},
    NUMBER = {8},
     PAGES = {2433--2441},
      ISSN = {0002-9939,1088-6826},
   MRCLASS = {65D30 (26B15 41A55)},
  MRNUMBER = {1459132},
MRREVIEWER = {C.\ Dagnino},
       DOI = {10.1090/S0002-9939-98-04454-2},
       URL = {https://doi.org/10.1090/S0002-9939-98-04454-2},
}

@article {MingXu07,
    AUTHOR = {Ming, Wang and Xu, Jinchao},
     TITLE = {Nonconforming tetrahedral finite elements for fourth order
              elliptic equations},
   JOURNAL = {Math. Comp.},
  FJOURNAL = {Mathematics of Computation},
    VOLUME = {76},
      YEAR = {2007},
    NUMBER = {257},
     PAGES = {1--18},
      ISSN = {0025-5718,1088-6842},
   MRCLASS = {65N30},
  MRNUMBER = {2261009},
MRREVIEWER = {Stephen\ W.\ Brady},
       DOI = {10.1090/S0025-5718-06-01889-8},
       URL = {https://doi.org/10.1090/S0025-5718-06-01889-8},
}

@incollection{Neilan20,
 author = {Neilan, Michael},
 title = {The {Stokes} complex: a review of exactly divergence-free finite element pairs for incompressible flows},
 booktitle = {75 years of mathematics of computation. Symposium celebrating 75 years of mathematics of computation, Institute for Computational and Experimental Research in Mathematics, ICERM, Providence, RI, USA, November 1--3, 2018},
 isbn = {978-1-4704-5163-9; 978-1-4704-5637-5},
 pages = {141--158},
 year = {2020},
 publisher = {Providence, RI: American Mathematical Society (AMS)},
 language = {English},
 doi = {10.1090/conm/754/15142},
 zbMATH = {7370293},
 Zbl = {1477.65236}
}

@article{PapanicolopulosZervosVardoulakis09,
 author = {Papanicolopulos, S.-A. and Zervos, A. and Vardoulakis, I.},
 title = {A three-dimensional {{\(C^1\)}} finite element for gradient elasticity},
 fjournal = {International Journal for Numerical Methods in Engineering},
 journal = {Int. J. Numer. Methods Eng.},
 issn = {0029-5981},
 volume = {77},
 number = {10},
 pages = {1396--1415},
 year = {2009},
 language = {English},
 doi = {10.1002/nme.2449},
 zbMATH = {5531208},
 Zbl = {1156.74382}
}

@article{SchumakerSorokinaWorsey09,
 author = {Schumaker, Larry L. and Sorokina, Tatyana and Worsey, Andrew J.},
 title = {A {{\(C^1\)}} quadratic trivariate macro-element space defined over arbitrary tetrahedral partitions},
 fjournal = {Journal of Approximation Theory},
 journal = {J. Approx. Theory},
 issn = {0021-9045},
 volume = {158},
 number = {1},
 pages = {126--142},
 year = {2009},
 language = {English},
 doi = {10.1016/j.jat.2008.04.014},
 zbMATH = {5558971},
 Zbl = {1176.41007}
}

@article {Sitaramachandra87,
    AUTHOR = {Sitaramachandra Rao, R.},
     TITLE = {A formula of {S}.~{R}amanujan},
   JOURNAL = {J. Number Theory},
  FJOURNAL = {Journal of Number Theory},
    VOLUME = {25},
      YEAR = {1987},
    NUMBER = {1},
     PAGES = {1--19},
      ISSN = {0022-314X,1096-1658},
   MRCLASS = {11M06 (40A25)},
  MRNUMBER = {871165},
MRREVIEWER = {Bruce\ C.\ Berndt},
       DOI = {10.1016/0022-314X(87)90012-6},
       URL = {https://doi.org/10.1016/0022-314X(87)90012-6},
}

@article{WorseyFarin87,
 author = {Worsey, A. J. and Farin, G.},
 title = {An {{\(n\)}}-dimensional {Clough}-{Tocher} interpolant},
 fjournal = {Constructive Approximation},
 journal = {Constr. Approx.},
 issn = {0176-4276},
 volume = {3},
 pages = {99--110},
 year = {1987},
 language = {English},
 doi = {10.1007/BF01890556},
 zbMATH = {4026035},
 Zbl = {0631.41003}
}

@article {WorseyPiper88,
    AUTHOR = {Worsey, A. J. and Piper, B.},
     TITLE = {A trivariate {P}owell-{S}abin interpolant},
   JOURNAL = {Comput. Aided Geom. Design},
  FJOURNAL = {Computer Aided Geometric Design},
    VOLUME = {5},
      YEAR = {1988},
    NUMBER = {3},
     PAGES = {177--186},
      ISSN = {0167-8396,1879-2332},
   MRCLASS = {65D05 (52-04)},
  MRNUMBER = {959603},
MRREVIEWER = {Shou\ Shan\ Jiang},
       DOI = {10.1016/0167-8396(88)90001-5},
       URL = {https://doi.org/10.1016/0167-8396(88)90001-5},
}

@article{Zenisek73,
title = {Polynomial approximation on tetrahedrons in the finite element method},
fjournal = {Journal of Approximation Theory},
journal = {J. Approx. Theory},
volume = {7},
number = {4},
pages = {334-351},
year = {1973},
issn = {0021-9045},
doi = {https://doi.org/10.1016/0021-9045(73)90036-1},
url = {https://www.sciencedirect.com/science/article/pii/0021904573900361},
author = {{\v Z}en{\'\i}{\v s}ek, Alexander},
}

@article{ZhaoSchillingerXu17,
 author = {Zhao, Ying and Schillinger, Dominik and Xu, Bai-Xiang},
 title = {Variational boundary conditions based on the {Nitsche} method for fitted and unfitted isogeometric discretizations of the mechanically coupled {Cahn}-{Hilliard} equation},
 fjournal = {Journal of Computational Physics},
 journal = {J. Comput. Phys.},
 issn = {0021-9991},
 volume = {340},
 pages = {177--199},
 year = {2017},
 language = {English},
 doi = {10.1016/j.jcp.2017.03.040},
 zbMATH = {6818920},
 Zbl = {1380.65293}
}
\end{document}